\documentclass[11 pt]{article}  
\usepackage[utf8]{inputenc}
\usepackage{amsmath}
\usepackage{amsfonts}
\usepackage{amssymb}
\usepackage{graphicx}
\usepackage{mathrsfs}
\usepackage{upref,amsthm,amsxtra,exscale}
\usepackage{stmaryrd}
\usepackage{cite}
\usepackage[colorlinks=true,urlcolor=blue,
citecolor=red,linkcolor=blue,linktocpage,pdfpagelabels,
bookmarksnumbered,bookmarksopen]{hyperref}

\usepackage{fullpage}

\usepackage{subcaption}
\usepackage{caption}
\usepackage{cleveref}

\usepackage{enumitem}

\def\N{\mathbb{N}}
\def\eps{\varepsilon}
\def\R{\mathbb{R}}

\def\cN{\mathcal{N}}

\def\cE{\mathcal{E}}

\def\cH{\mathcal{H}}
\def\ast{*}
\def\weakto{\rightharpoonup}

\newtheorem{theorem}{Theorem}[section]
\newtheorem{corollary}[theorem]{Corollary}
\newtheorem{definition}[theorem]{Definition}
\newtheorem{lemma}[theorem]{Lemma}
\newtheorem{proposition}[theorem]{Proposition}

\numberwithin{equation}{section}

\theoremstyle{definition}
\newtheorem{remark}[theorem]{Remark}

\usepackage{color}
\usepackage[dvipsnames]{xcolor}

\title{Uniqueness and nondegeneracy of least-energy solutions to fractional Dirichlet problems}
\author{
Abdelrazek Dieb,
Isabella Ianni\footnote{I. Ianni is  supported by INDAM - GNAMPA,   {\sl Fondi Ateneo - Sapienza},   PRIN.},
\& Alberto Saldaña \footnote{A. Saldaña is supported  by 
the 2021 Visiting Professor Program of La Sapienza University (Italy), by 
CONACYT grant A1-S-10457 (Mexico), and by UNAM-DGAPA-PAPIIT grant IA100923 (Mexico).
}}
\date{}

\begin{document}

\maketitle

\abstract{We prove the uniqueness and nondegeneracy of least-energy solutions of a fractional Dirichlet semilinear problem in  sufficiently large balls and in more general symmetric domains.  Our proofs rely on uniform estimates on growing domains, on the uniqueness and nondegeneracy of the ground state of the problem in $\mathbb R^N$, and on a new symmetry characterization of the eigenfunctions of the linearized eigenvalue problem in domains which are convex in the $x_1$-direction and symmetric with respect to a hyperplane reflection.}

\bigskip
\bigskip

\noindent\textsc{Keywords:} Fractional Laplacian, uniqueness, nondegeneracy, symmetry of eigenfunctions.
\medskip

\noindent\textsc{MSC2010:}
35S15 ·  
35A02 ·  
35B40.  
\medskip

\noindent\textsc{Data availability statement:}  the manuscript has no associated data.
\medskip

\section{Introduction and main results}
In the last decades many results have been obtained concerning  the uniqueness and nondegeneracy of positive solutions for the following  local semilinear elliptic problem
\begin{align}\label{c:intro}
    -\Delta u +\lambda u&= u^{p}\quad \text{ in }\Omega,\qquad u=0\quad \text{ on }\partial \Omega,
\end{align}
where $\Omega$ is a bounded  smooth domain in $\mathbb R^N$, $N\geq 2$, $\lambda \in \mathbb R$ and $p>1$.

In the seminal paper \cite{GNN1979}, uniqueness has been obtained when $\Omega$ is a ball and in the case $\lambda=0$ by a direct scaling argument and ODE techniques, after observing that the solution is radially symmetric. 
The case $\lambda\neq 0$ is more involved, it has  required the contribution of many authors
and it is spread in several papers (see \cite{NiNussbaumCPAM1985, KwongLiTAMS1992, ZhangCPDE1992, SrikanthDiffIntEq1993, AdimurthiYadavaARMA1994, AftalionPacellaJDE2003}), where the  uniqueness  result is proved when $\Omega$ is a ball and in the 
full range of values of  $p$ and $\lambda$ for which existence holds, namely for any  $p\in (1,2^*-1)$, ($2^*:=\frac{2N}{N-2}$ if $N>2$, $2^*:=\infty$ if $N=2$)  and any $\lambda> -\lambda_1(\Omega)$, where $\lambda_1(\Omega)$
is the
first Dirichlet eigenvalue of the Laplace operator $-\Delta$ in $\Omega$.

The uniqueness property depends strictly on the domain $\Omega$ and it is influenced by its topology and geometry, for instance it does not hold when $\Omega$ is an annulus or a suitable dumb-bell domain (see e.g.\cite{DancerJDE1988,DancerJDE1990}).
Nevertheless uniqueness results for positive solutions of~\eqref{c:intro} have been obtained in some domains $\Omega$ different from the  ball, especially in the case $\lambda=0$ (see \cite{ZouPisa1994,DancerMa2003, DamascelliGrossiPacellaAIHP1999,GrossiADE2000,lin1994,demak}), and some when $\lambda\neq 0$
(see \cite{DancerRockyMountain1995,McKennaPacellaPlumRothJDE2009,McKennaPacellaPlumRoth2012}). 
In \cite{ZouPisa1994} the domain $\Omega$ is a suitable perturbation of the ball, while  
\cite{DancerMa2003, DamascelliGrossiPacellaAIHP1999,GrossiADE2000} assume  $\Omega$ to be symmetric and convex with respect to $N$ orthogonal directions.
During the eighties the convexity of the domain $\Omega$ was conjectured to be a sufficient condition for uniqueness of positive solutions of problem~\eqref{c:intro} (see \cite{DancerJDE1988, KawohlLectureNotes1985}), and in the case $\lambda=0$ this conjecture was first partially solved by Lin (see \cite{lin1994}) in dimension $N=2$, who proved it  for \emph{least-energy solutions}, and only recently the proof has been extended to general positive solutions, for values of $p$ sufficiently close to $2^*$ (see \cite{demak, GrossiIanniLuoYan} for the case $N=2$, see  \cite{LiWeiZou} for $N\geq 3$). For $\lambda\neq 0$ the problem becomes even more difficult, the conjecture is still completely open and the available  uniqueness results are quite few. Apart form the case of the ball, the only  uniqueness results in the case $\lambda\neq 0$  of which we are aware are on  suitable \emph{symmetric large domains} $\Omega$ when  $\lambda>0$ (see \cite{DancerRockyMountain1995}), and when  $\Omega$ is the unit planar square, $\lambda\neq 0$, and either $p=2$ or $p=3$ (see   \cite{McKennaPacellaPlumRothJDE2009,McKennaPacellaPlumRoth2012} where uniqueness is obtained via a computer-assisted proof).

 Another interesting and related problem is the study of the spectral properties of the linear operator obtained by linearizing the equation~\eqref{c:intro} about a solution $u$, namely the study of the eigenvalues $\mu$ and the associated eigenfunctions $v$ of
\begin{align}\label{linearizzatoLocale}
-\Delta v +\lambda v-p|u|^{p-1}v=\mu v \quad\mbox{ in }\Omega,\qquad
v=0 \quad\mbox{ on }\partial \Omega.
\end{align}
In particular a solution $u$ is said to be nondegenerate  when $\mu=0$ is not an eigenvalue for~\eqref{linearizzatoLocale}. 

We  stress that, for problem~\eqref{c:intro}, the question of the uniqueness is strictly related to  the problem of the nondegeneracy  of the solutions. Indeed, as shown in \cite{lin1994} (see also \cite{DamascelliGrossiPacellaAIHP1999}), uniqueness for the positive solutions of~\eqref{c:intro} may be derived from the nondegeneracy via a continuation argument in the exponent $p$ based on the implicit function theorem, since both  uniqueness and nondegeneracy hold for $p$ close enough to $1$ (see \cite{DancerMa2003,DamascelliGrossiPacellaAIHP1999,lin1994}). Both the proofs of the nondegeneracy in \cite{lin1994} and in \cite{DamascelliGrossiPacellaAIHP1999} are then based on the study of 
the properties of the  eigenfunctions $v$ of~\eqref{linearizzatoLocale} corresponding to the zero eigenvalue $\mu=0$.
In particular, in \cite{DamascelliGrossiPacellaAIHP1999}, a purely PDE approach based on the maximum principle is exploited in order to derive symmetry properties of the eigenfunctions of the linearized problem.

In this work, we study  uniqueness, nondegeneracy and spectral properties in the nonlocal version of~\eqref{c:intro}. 
We consider the following nonlocal exterior Dirichlet problem
\begin{align}
\label{Peps}
(-\Delta)^s u +\lambda u=u^p \quad\mbox{ in }\Omega,\qquad
u=0 \quad\mbox{ in }\mathbb R^N \setminus \Omega,
\end{align}
where $\Omega\subset\mathbb R^N$, $N\geq 1$, is a smooth bounded domain, $s\in (0,1)$, $\lambda\in\mathbb R$, $p>1$ and
 $(-\Delta)^s$ denotes the fractional Laplacian given by
\begin{align}\label{operator}
(-\Delta)^{s}u(x):=c_{N,s}\, 
p.v.\int_{\mathbb{R}^{N}}{\frac{u(x)-u(y)}{|x-y|^{N+2s}}\, dy},\qquad c_{N,s}
:=4^{s}\pi^{-\frac{N}{2}}\frac{\Gamma(\frac{N}{2}+s)}{\Gamma(2-s)}s(1-s).
\end{align}
We also consider the eigenvalue problem for the associated linearized equation at a non-negative solution $u$ of~\eqref{Peps}
\begin{equation}\label{linearizedEqINTRO}
(-\Delta)^{s} v+\lambda v-pu^{p-1}v=\mu v \quad \text{ in }\Omega,\qquad v=0\quad  \text{ in }\mathbb R^N\backslash\Omega.
\end{equation}
Let us observe that if $u$ is a non-negative solution of~\eqref{Peps}, then by the fractional strong maximum principle it follows that
\[u>0 \qquad \mbox{ in }\Omega.\]
Though we are not interested in the existence of solutions for~\eqref{Peps},  let
us point out that, as simple computations show,  non-negative solutions of~\eqref{Peps}  do not exist if $\lambda\leq -\lambda_1^s(\Omega)$, where $\lambda_1^s(\Omega)$
is the
first eigenvalue of the fractional Laplace operator $(-\Delta)^s$ with exterior homogeneous Dirichlet
boundary conditions in $\Omega$.
Hence it makes sense to consider only the case \[\lambda>-\lambda_1^s(\Omega).\]
Furthermore, letting $2^\star_s$ be the critical fractional Sobolev exponent given by
\[2^\star_s=\frac{2N}{N-2s} \ \ \mbox{ if }2s<N \quad\mbox{ and }\quad 2^\star_s=+\infty\ \ \mbox{ if }2s\geq N( =1),\]
we recall that non-existence holds when $\Omega$ is starshaped, $\lambda\geq 0$, and the exponent $p\geq 2^\star_s-1$, $N>2s$ (via Pohozaev identity, cfr. \cite{Ros-OtonSerra2014}). On the other hand, existence of  non-negative  solutions for~\eqref{Peps} can be obtained by standard variational methods, for any $\lambda> -\lambda_1^s(\Omega)$ and $p\in(1, 2^\star_s-1)$ in any bounded domain $\Omega$ (by the compact embedding of the fractional Sobolev space  $\mathcal H^s_0(\Omega)\subset L^{p+1}(\Omega)$, see e.g. \cite{servadei2013variational})
 and,  for $p\geq 2^\star_s-1$ there may be existence depending on the domain $\Omega$ (e.g when $\Omega$ is an annulus, $\mathcal H^s_{0,rad}(\Omega)\subset L^{p+1}(\Omega)$ is compact if $s>\frac{1}{2}$ and $N\geq 2$, via fractional Strauss inequality cfr. \cite{deNapoli,ChoOzawa}).

Let us observe that  non-negative  solutions of~\eqref{Peps} are radially symmetric when $\Omega$ is a ball (see, e.g., \cite{JW14}); nevertheless, the usual methods  used to get uniqueness   in the local case,
such as shooting methods and other ODE techniques, are not available in the fractional setting. Furthermore, most of the essential tools usually used  to derive  nondegeneracy in the local case, such as Courant’s nodal theorems or Hopf Lemmas for sign-changing solutions, are also not available, which complicates the analysis of the linearized
associated problem~\eqref{linearizedEqINTRO} and makes the question of the nondegeneracy and uniqueness for the nonlocal problem~\eqref{Peps}
 very challenging, even in the case in which  $\Omega$ is a ball! To the authors's knowledge,  the only results available in the literature  for uniqueness and nondegeneracy of non-negative solutions of~\eqref{Peps}
 are in the asymptotic regimes  as $s\rightarrow 1$ or as $p\rightarrow 1$ (see \cite{DIS22}, see also \cite{FranzinaLicheri22} where in the case $\lambda=0$ the result for  $p\rightarrow 1$ has been also obtained). As a consequence of these results in the asymptotically linear regime (\emph{i.e.} $p\rightarrow 1$), one can then adapt to the nonlocal framework similar arguments as in \cite{lin1994,DamascelliGrossiPacellaAIHP1999}, and  show that if the nondegeneracy is achieved, then also the uniqueness may be recovered in the full range of $p$'s, by a continuation argument in $p$, based on the implicit function theorem (see Remark~\ref{rmk:nondegeneracyImpliesUniquenes}).

In this paper we prove nondegeneracy and uniqueness for \emph{least-energy solutions} of~\eqref{Peps}. Our main result is the following.
\begin{theorem}\label{theorem:uniquenessAndNondegeneracyB}
Let $N\geq 2$ and let $\Omega=B_R\subset\mathbb R^N$ be a ball of radius $R>0$. Let $s\in (0,1)$,  $p\in(1,2^\star_s-1),$ and $\lambda> 0$.  
There exists $R_0:=R_0(p,\lambda, s)\geq 1$ such that, if $R\geq R_0$, then problem~\eqref{Peps} has a unique least-energy solution and it is nondegenerate.
\end{theorem}

We point out that in Theorem~\ref{theorem:uniquenessAndNondegeneracyB}
we can consider any $\lambda>0$. 
Moreover, differently from  the previous results in  \cite{DIS22}, here we cover both the full fractional range $s\in (0,1)$ and the complete range of superlinear powers  $p$'s for which there is existence. 

In order to achieve the nondegeneracy in Theorem~\ref{theorem:uniquenessAndNondegeneracyB}, a key ingredient is given by the following symmetry result for the linearized equation~\eqref{linearizedEqINTRO}, which is of independent interest.  In particular, it characterizes the reflectional symmetry of the eigenfunctions of~\eqref{linearizedEqINTRO}, whenever the domain is convex in the $x_1$-direction and symmetric with respect to the hyperplane $\left\lbrace x\in \mathbb R^N:\, x_1=0\right\rbrace$. 

\begin{theorem}\label{sym} 
Let $N\geq 1$ and assume  that $\Omega\subset\mathbb R^N$ is a smooth bounded domain which is convex in the $x_1$-direction and symmetric with respect to the hyperplane $H:=\left\lbrace x\in \mathbb R^N:\, x_1=0\right\rbrace$.  Let $s\in (0,1)$, $\lambda>-\lambda_1^s(\Omega)$, 
$p\in(1,2^\star_s-1]$, 
and $u$ be a non-negative solution of~\eqref{Peps}. 
Let $\mu_2$ be the second eigenvalue of~\eqref{linearizedEqINTRO} and assume that $\mu_2\leq 0$.
Then any eigenfunction $v$ solving
\begin{equation}\label{eq:secondEigen}
(-\Delta)^{s} v+\lambda v-pu^{p-1}v=\mu_2 v \quad \text{ in }\Omega,\qquad v=0\quad  \text{ in }\mathbb R^N\backslash\Omega,
\end{equation}
is (sign-changing and) symmetric with respect to $x_1$.
\end{theorem}
We stress that this symmetry result holds true under the assumption that $u$ is \emph{any non-negative solution} of~\eqref{Peps}, differently from Theorem~\ref{theorem:uniquenessAndNondegeneracyB} where $u$ is required to be a least-energy solution. The assumption on the exponent $p$ implies that $u$ is bounded (see \cite[Proposition 9.1]{RosOtonSerraValdinoci}), so that \eqref{linearizedEqINTRO} admits a sequence of eigenvalues $\mu_i$ (see Section \ref{section:linearizedEigenvalueProblem}). Furthermore we can cover the full range of $\lambda$'s  for which~\eqref{Peps} admits a non-negative solution $u$, in particular  also $\lambda\leq 0$ is admissible. Dimension $N=1$ is also considered.

We point out that the assumption $\mu_2\leq 0$ is not restrictive when one aims at proving the nondegeneracy of $u$, indeed we can show that $\mu_1<0$ (see Section~\ref{sec:pre}), hence $u$ is clearly nondegenerate if $\mu_2>0$. As a consequence the case $\mu_2\leq 0$
is exactly the one to consider when studying a possible degeneracy of $u$.

As a special case, when $\Omega=B$,  Theorem~\ref{sym} implies the radial symmetry of any second eigenfunction for the eigenvalue problem~\eqref{eq:secondEigen}, with $\mu_2\leq 0$.

\begin{corollary}\label{corollary:radialKernel}
Let $N\geq 1$ and let $\Omega\subset\mathbb R^N$ be a ball. 
Let $s\in (0,1)$, $\lambda>-\lambda_1^s(\Omega)$, 
$p\in(1,2^\star_s-1]$ and let $u$ be a non-negative solution of~\eqref{Peps}.
If $\mu_2\leq 0$, then any nontrivial solution $v$ to~\eqref{eq:secondEigen} is (sign-changing and) radial. Moreover, $v$ changes sign at most two times if $N\geq 2$, and exactly once if $N=1$.
\end{corollary}

Theorem \ref{sym} and Corollary \ref{corollary:radialKernel} are the counterpart, in the nonlocal framework, of classical symmetry results for the eigenfunctions of \eqref{linearizzatoLocale} associated to eigenvalues $\mu\leq 0$ (see  \cite{DamascelliGrossiPacellaAIHP1999, Babin,AftalionPacellaJDE2003,LinNi1988}).

Concerning the proof of Theorem~\ref{sym}, we observe that
showing symmetry properties for sign-changing solutions is in general nontrivial and does not follow from a moving-planes type argument.
We took some ideas from the proof of \cite[Theorem 2.1]{DamascelliGrossiPacellaAIHP1999}, where
in the local case a similar result has been derived by the use of the maximum principle and the Hopf lemma. 
However, a main difficulty in the nonlocal setting arises from the fact that strong maximum principles and Hopf Lemmas are not available for  general 
 sign-changing solutions.

Our proof  exploits a nonlocal version of the strong maximum principles and of the Hopf Lemma, which hold for antisymmetric solutions (see \cite{JarohsWeth2016,FallJarohs2015}).

More precisely, we start from a nontrivial  solution $v$ of the linearized eigenvalue problem~\eqref{eq:secondEigen}  and we first rearrange it by considering its polarization $Pv$ with respect to the hyperplane $H$.  Following some arguments similar as in \cite{Benedikt}, based on the introduction of a new variational characterization for the second eigenfunctions, we can show that $Pv$ also solves~\eqref{eq:secondEigen}.  
We then define the antisymmetric eigenfunction $w$ given by the difference between $Pv$ evaluated at a point and at the point obtained by reflection with respect to the hyperplane $H$.
By construction, $w$ does not change sign in the half-domain $\{x\in \Omega\ :\ x_1>0\}$. 
As a consequence, $w$ satisfies  the nonlocal antisymmetric  maximum principle 
 and Hopf Lemma. Hence,  by a nonlocal integration by parts formula (from \cite{Ros-OtonSerra2014}), we  can conclude that $w=0$, namely that $Pv$ is symmetric, from which the symmetry of $v$ finally follows.\\

The symmetry property for the solutions of the linearized equation given in Theorem~\ref{sym} is then used to get the nondegeneracy result stated in Theorem~\ref{theorem:uniquenessAndNondegeneracyB}.
Indeed let us observe that if we restrict to consider non-negative solutions $u$ of~\eqref{Peps} with Morse index equal to $1$ (e.g. least-energy solutions, as  in Theorem~\ref{theorem:uniquenessAndNondegeneracyB}), then 
saying that $u$ is degenerate, means that  $\mu_2=0$ and that the eigenvalue problem~\eqref{eq:secondEigen} admits a nontrivial solution $v$, which is \emph{radially symmetric} by Theorem~\ref{sym}. In our proof we then achieve the conclusion by 
growing the domain and combining this symmetry result with the nondegeneracy and uniqueness result proved in \cite{FLS16} for ground state solutions of 
\begin{align}\label{Q:intro}
(-\Delta)^s Q +\lambda Q = Q^p \quad\mbox{ in }\mathbb R^N.
\end{align}
Indeed the result in \cite{FLS16} implies in particular that the solutions of the associated linearized problem at $Q$ are \emph{not radial}.
As a consequence we can deduce that $\mu_2>0$, provided  the radius of the ball is large enough, thus getting the nondegeneracy of $u$. We remark that some compactness in the subspace of radial functions is needed in the asymptotic procedure, for this reason in  Theorem~\ref{theorem:uniquenessAndNondegeneracyB} we restrict to  $N\geq 2$ (see Remark \ref{rmk:compactembeddings}).
As already observed, uniqueness then follows from nondegeneracy, anyway we prove uniqueness  with a different and more direct argument, again based on asymptotic methods and on the uniqueness and nondegeneracy result in \cite{FLS16} (see Section~\ref{subsec:UniqNondegLarge}).\\

We stress that, if one was able to exhibit the existence of a non-symmetric second eigenfunction, then Theorem~\ref{sym} would imply that $\mu_2>0$ necessarily. Hence, since we know that $\mu_1<0$, the nondegeneracy   for \emph{any} non-negative solutions of~\eqref{Peps} would immediately follow and, from it, the uniqueness! 

In the case when $\Omega=B$ in a previous version of this paper (cf. \cite{DIS23}) we erroneously considered as a nonradial candidate for the second eigenfunction, the function obtained from a radial second eigenfunction $w$ by taking its polarization $P_a w$  with respect to the hyperplane $H_a=\{x\in\mathbb R^N: x_1=a\}$, for  suitable $a\neq 0$. Sadly when $a\neq 0$ it cannot be proven that $P_a w$ is a second eigenfunction for the nonautonomous linear problem~\eqref{eq:secondEigen} (see Remark~\ref{rkk:polarization_only_with_a=0} for more details). We point out that the same mistake has been made in the papers \cite{PD23,FW23preprint,LiSong24preprint},  where the same function $P_aw$, $a\neq 0$ has been considered (in $B$ and in the whole $\mathbb R^N$) as a candidate for the nonradial second eigenfunction.\\

We believe that the symmetry result in Theorem~\ref{sym} may be a main tool to get nondegeneracy and uniqueness results in many situations, also different from the one considered in Theorem~\ref{theorem:uniquenessAndNondegeneracyB}, for instance in the case when $\Omega$ is convex, $\lambda=0$ (i.e. the fractional Lane-Emden problem), and $p=2^\star_s-1-\varepsilon$, for $\varepsilon>0$ small. We will discuss these results in a forthcoming paper.\\

Finally we notice that the proof of Theorem~\ref{theorem:uniquenessAndNondegeneracyB} can be extended to obtain a more general result. To be more precise, let us consider a domain $D$ that satisfies the following.
\begin{itemize}
\item[$(\mathcal S)$] $D \subset\mathbb R^N$ ($N\geq 2$) is a Lipschitz bounded $G$-invariant domain and  
convex in the $x_i$-direction,
for any $i=1,\ldots,N,$
\end{itemize}
where $G$ is the group of isometries
$G:=O(m_1)\times \ldots \times O(m_\ell)\subseteq O(N), $
with $\ell\geq 1$, $m_i\geq 2$ for $i=1,\ldots,\ell$ such that $\sum_{i=1}^\ell m_i =N,$ and $O(m_i)$ denotes the group of linear isometries of $\R^{m_i}$. 

\begin{theorem}\label{theorem:nondegeneracyLARGE}\label{theorem:uniquenessLARGE}
Let $s\in (0,1)$, $\lambda>0$ and $p\in(1,2^\star_s-1)$. Let $N\geq 2$, let $D\subset\mathbb R^N$ satisfy $(\mathcal S)$, and let 
\begin{align*}
\Omega= D_{R}:=R D=\{ x \in \mathbb R^N\ :\ R^{-1}x\in D \}\ \mbox{  for }R\geq 1.   
 \end{align*}
Then there exists $R_0=R_0(D,p,\lambda, s)\geq 1$ such that, for $R\geq R_0$, the problem~\eqref{Peps} has a unique least-energy solution and it is nondegenerate.
\end{theorem}

Clearly, if $D$ is a ball, then it satisfies $(\mathcal S)$, hence Theorem~\ref{theorem:uniquenessAndNondegeneracyB} is a corollary of Theorem~\ref{theorem:nondegeneracyLARGE}. 

On the other hand, there are domains that satisfy $(\mathcal S)$ which are not balls; for instance, the ellipsoids 
\begin{align*}
D:=\left\{ x=(x_1,\ldots,x_\ell)\in \R^{m_1}\times \ldots\times \R^{m_\ell}\::\: \frac{|x_1|^2}{a_1}+\ldots+\frac{|x_\ell|^2}{a_\ell}\leq 1 \right\},
\end{align*}
where $a_i>0$,  $m_i\geq 2$, $\ell\geq 2$ and $N=\sum_{i=1}^\ell m_i(\geq 4)$.

We remark that, for the local problem~\eqref{c:intro}, a result in the spirit of the one in Theorem~\ref{theorem:nondegeneracyLARGE} can be found in \cite{DancerRockyMountain1995}.  

Furthermore, by choosing $R:=\eps^{-\frac{1}{2s}}$, Theorem~\ref{theorem:uniquenessLARGE} and a change of variables immediately yields a uniqueness and nondegeneracy result for singularly perturbed problems.
\begin{corollary}
Let $s\in (0,1)$, $\lambda>0$ and $p\in(1,2^\star_s-1)$.   If $N\geq 2$ and $D\subset\mathbb R^N$ satisfies $(\mathcal S)$, then there is $\eps_0>0$ such that, for $\eps\in(0,\eps_0)$, the  problem
\begin{align*}
\varepsilon^{2s}(-\Delta)^s u+\lambda u=u^p\quad \text{ in }D,\qquad 
u=0 \quad \text{ in }\mathbb R^N \setminus D,
\end{align*}
has a unique least-energy solution and it is nondegenerate.
\end{corollary}

\medskip

The paper is organized as follows.   In Section~\ref{sec:pre} we define our notation and present some preliminary results.   Section~\ref{sec:symm} is devoted to the proof of the symmetry result for the linearized problem Theorem~\ref{sym}.    Finally, in Section~\ref{section:uniquenessLargeDomain} we show  Theorem~\ref{theorem:uniquenessLARGE} regarding
the uniqueness and nondegeneracy result in large symmetric domains. Theorem~\ref{theorem:uniquenessAndNondegeneracyB} is a special case of Theorem~\ref{theorem:uniquenessLARGE}.

\subsection*{Acknowledgements}This work was conceived and partially completed while the authors were mutually visiting the following institutions, which they wish to thank for the warm hospitality:
 \emph{SBAI Department} (\emph{Sapienza} University, Rome, Italy), \emph{Instituto de Matemáticas} (UNAM, Mexico City, Mexico), \emph{KIT Math Department} (Karlsruhe, Germany), \emph{ICTP} (Trieste, Italy).
 The authors also would like to thank M. Fall, E. Parini, and T. Weth for useful discussions.

\section{Notation and preliminary results}\label{sec:pre}

For a function $w$ we denote 
\begin{equation}\label{defpositiveNegativaPart}
    w^+:=\max\{w,0\}\geq 0 \qquad\mbox{ and }\qquad w^-:=\max\{-w,0\}\geq 0 ,
\end{equation}
hence $w=w^+-w^-$.  For an open set $D\subset \R^N$,  $N\geq 2$, we define 
\begin{align*}
&|w|_{q;D}:=\left(\int_{D}|w|^q\, dx\right)^\frac{1}{q}\quad \text{ for }q\in[1,\infty),\qquad |w|_{\infty;D}:=\sup_{D}|w|.
\end{align*}
We sometimes omit $D$ if $D=\R^N$, namely, $|w|_{q}:=|w|_{q;\R^N}$.

Let $s\in (0,1)$, $\Omega\subset \R^N$ be a smooth open bounded set,  and let $\lambda>-\lambda_1^s(\Omega)$, where $\lambda_1^s(\Omega)$ denotes the first Dirichlet eigenvalue of $(-\Delta)^s$ in $\Omega.$  We define $H^{s}(\R^N)$ the fractional Sobolev space,
\begin{align*}
H^s(\R^N)=\left\lbrace w \in
L^2(\R^N):\frac{|w(x)-w(y)|}{|x-y|^{\frac{N}{2}+s}} \in
L^2(\R^N\times\R^N)\right\rbrace ,
\end{align*}
endowed with the (equivalent) norm
\begin{align}\label{norm:def}
\|w\|^2:= \lambda|w|_2^2+\cE(w,w),
\end{align}
where, for $w,v\in H^s(\R^N)$, 
\begin{align*}
\cE(w,v):=\dfrac{c_{N,s}}{2}\iint_{\R^N\times \R^N}\dfrac{(w(x)-w(y))(v(x)-v(y))}{|x-y|^{N+2s}}dx\,dy
\end{align*}
and $c_{N,s}:=4^{s}\pi^{-\frac{N}{2}}\frac{\Gamma(\frac{N}{2}+s)}{\Gamma(2-s)}s(1-s)$ is a normalizing constant. 

For a given smooth open bounded set $\Omega\subset \R^N$, let
\begin{align*}
\cH^s_0(\Omega)=\left\{ w \in H^{s}(\R^N):\, \, w=0 \, \text{ in } \R^N\backslash \Omega \right\}.    
\end{align*}
Notice that $\cH^s_0(\Omega)$ is a Hilbert space with the norm $\|\cdot\|$.  We also use $H^1_0(\Omega)$ to denote the usual Sobolev space of weakly differentiable functions with zero trace. 

For $m\in \N_0$, $\sigma\in{(0,1]}$, $s=m+\sigma$, we write $C^s(\Omega)$ (resp. $C^s(\overline{\Omega})$) to denote the space of $m$-times continuously differentiable functions in $\Omega$ (resp. $\overline{\Omega}$) whose derivatives of order $m$ are locally $\sigma$-H\"older continuous in $\Omega$ {(or Lipschitz continuous if $\sigma=1$)}.  We use $[\,\cdot\,]_{C^\sigma(\Omega)}$ for the Hölder seminorm in a domain $\Omega$, namely, 
\begin{align*}
        [w]_{C^\sigma(\Omega)}:=\sup_{\substack{x,y\in \Omega\\x\neq y}}\frac{|w(x)-w(y)|}{|x-y|^\sigma}
\end{align*}
and $\|w\|_{C^s(\Omega)}:= \sum_{|\alpha|\leq m}|\partial^\alpha w|_{\infty;\Omega}+\sup_{|\alpha|=m}[\partial^\alpha w]_{C^\sigma(\Omega)}$ is the usual Hölder norm in $C^s(\Omega)$.

\subsection{The linearized problem}\label{section:linearizedEigenvalueProblem}
Let $\Omega\subset\mathbb R^N$, $N\geq 1$ be a smooth open bounded set and let us consider the problem 
\begin{align} \label{P}
(-\Delta)^s u +\lambda u=|u|^{p-1}u \mbox{ in }\Omega,\qquad u=0 \mbox{ in }\mathbb R^N \setminus \Omega,
\end{align}
where $s\in (0,1)$, $p\in (1, 2^\star_s-1]$ and $\lambda>-\lambda_1^s(\Omega)$.\\

We introduce the  eigenvalue problem associated to the linearized equation at a solution $u$ of~\eqref{P}, namely
\begin{equation}\label{eigenvalueProblematu}
(-\Delta)^{s} \phi+\lambda \phi-p|u|^{p-1}\phi=\mu\phi \quad\text{ in }\Omega,
\qquad
\phi=0 \quad\text{ in }\mathbb R^N\setminus\Omega,
\end{equation}
It is well known that~\eqref{eigenvalueProblematu} admits an increasing sequence of eigenvalues $\mu_i\in \mathbb R$. We denote by $\phi_i$ the associated eigenfunctions. The eigenvalues $\mu_i$ can be variationally characterized using  the bilinear form  $\mathcal{B}:H^s(\mathbb R^N)\times H^s(\mathbb R^N)\to \R$  defined as
\begin{equation}\label{BilinearForm}\mathcal{B}(\varphi,\psi):=\cE(\varphi,\psi)+\lambda\int_{\Omega}\varphi \psi \,dx-\int_{\Omega} p|u|^{p-1}\varphi \psi\,dx,
\end{equation} 
in particular let us recall  that
\begin{align}\label{variationalmu1}
\mu_{1}:=\min \left\{\frac{\mathcal{B}(\varphi,\varphi)}{\int_{\Omega}\varphi^{2}dx}\ : \ \varphi\in  \cH_0^{s}(\Omega)\setminus\{0\} \right\}
\end{align}
and
\begin{equation}
\label{variationalCharctmu2}
\mu_2:=\min\left\{\frac{\mathcal B(\varphi,\varphi)}{\int_{\Omega}\varphi^2dx}\ :\ \varphi\in \mathcal H^s_0(\Omega)\setminus\{0\},\ \int_{\Omega}\phi_1 \varphi =0\right\}.
\end{equation}
Furthermore minimizers of~\eqref{variationalmu1} and~\eqref{variationalCharctmu2} are the associated eigenfunctions, i.e. solutions to the eigenvalue problem~\eqref{eigenvalueProblematu}  with $\mu=\mu_1$ and $\mu_2$ respectively (by arguing as in \cite{servadei2013variational}, for example).

\begin{definition}
A solution $u$ of~\eqref{P} is said to be nondegenerate if $\mu=0$ is not an eigenvalue for~\eqref{eigenvalueProblematu}. Namely if the linearized problem at $u$:
\begin{equation}\label{linearizedEquationSection}
(-\Delta)^{s} v+\lambda v-p|u|^{p-1}v=0\quad\text{ in }\quad \Omega,\qquad v=0 \quad\text{ in }\quad \mathbb R^N\setminus\Omega,
\end{equation}
admits only the trivial solution $v\equiv 0$ in $\R^N$.
\end{definition}

\begin{definition}
The Morse index $m(u)$ of a solution $u$ of~\eqref{P} is the number of the strictly negative eigenvalues of 
\eqref{eigenvalueProblematu}
counted with multiplicity.
\end{definition}

The following equivalence can be easily proven.
\begin{lemma}\label{lemma:equivalenceMorse}
  $m(u)$ coincides with   
  the maximal dimension of a subspace of $\cH_0^s(\Omega)$ where the quadratic form $\varphi\mapsto {\mathcal B}(\varphi,\varphi)$ is negative definite. 
\end{lemma}
We will also need the following properties.
\begin{lemma}\label{eg vl neg}
Let $(\phi_1,\mu_1)$ and $(\phi_2,\mu_2)$ be the first and second eigenpair of~\eqref{eigenvalueProblematu} respectively, where $u$ is a non-negative solution of~\eqref{P}. Then, $\phi_1$ does not change sign, $\phi_2$ is sign-changing  and $\mu_2>\mu_1$. Furthermore $m(u)\geq 1$, namely
\begin{equation}\label{firstEigenNegative}
\mu_1<0.\end{equation}
\end{lemma}
\begin{proof}
The first assertions are well known and may be proved in a standard way, using the  variational characterization of $\mu_1$ in~\eqref{variationalmu1} and $\mu_2$ in~\eqref{variationalCharctmu2}.
Next, we show that $\mu_1<0$.
Using $u$ as a test function in the weak formulation of~\eqref{eigenvalueProblematu} and $\phi_1$ as a test function in~\eqref{P}, we obtain
\begin{align*}
\mu_{1}\int_{\Omega}u\phi_1=(1-p)\int_{\Omega}u^p\phi_1,
\end{align*} 
from which the conclusion follows recalling that $p>1$ and that $u>0$ and $\phi_1>0$ in $\Omega$.
\end{proof}
\begin{remark}[Symmetry and monotonicity of $\phi_1$]
\label{remark:symmmetryPhi1}
Since  $\phi_1$  is non-negative, when the domain $\Omega$  is convex in the $x_1$-direction and symmetric with respect to the hyperplane $H:=\left\lbrace x\in \mathbb R^N:\, x_1=0\right\rbrace$, then 
      $\phi_1$ has reflectional symmetry with respect to $H$ and it is strictly monotone decreasing  in $|x_1|$ for $x=(x_1,\ldots,x_N)\in\Omega$ (this follows from moving plane arguments, see  \cite[Corollary 1.2]{JW14}).
\end{remark}
In the radial case we also know
\begin{lemma}\label{lemma:SignAlongBoundary}
Let $\Omega=B\subset \mathbb R^N,$ be a ball of radius $R>0$ and let $\phi_{2,rad}\in\mathcal H^s_0(\Omega)\setminus\{0\}$ be the second radial  eigenfunction of  the linearized problem~\eqref{eigenvalueProblematu}. Then $r\in[0,R]\mapsto \phi_{2,rad}(r)$  changes sign at most $2$ times. If $N=1$ it changes sign exactly once.
\end{lemma}
\begin{proof}
We know that the eigenfunction $\phi_{2,rad}$ changes sign. Similarly as in \cite{FLS16} we can extend the eigenfunction $\phi_{2,rad}$ in $\mathbb (\mathbb R^{N+1})^+$ and,  using a variational argument in the spirit of Courant’s nodal domain theorem for the extended eigenvalue problem, we deduce that  the extension of $\phi_{2,rad}$ has $2$ nodal domains. From the bound on the number
of nodal domains for the extension, since we are in a radial setting,  we thus derive the bound for the number of sign changes of $\phi_{2,rad}$.  
\end{proof}

\subsection{Least-energy solutions}\label{section:leastEnergyBounded}
\begin{definition}\label{def:leastEnergyBounded}
We say that a solution $u$ of 
\eqref{P} is a least-energy solution 
if
\begin{align}\label{ueps:def}
J(u)=\inf_{\cN(\Omega)}J=:c_\Omega,
\end{align}
where  $J:H^s(\R^N)\to\R$ is the energy functional given by
\begin{align}\label{def:functionalJ}
    J(w):=\frac{1}{2}\cE(w,w)+ \frac{\lambda}{2}|w|_2^2-\frac{|w|^{p+1}_{p+1}}{p+1},
\end{align}
and the set $\cN(\Omega)$ is the Nehari manifold
\begin{align}\label{Nehari:def}
\cN(\Omega):=\{w\in \cH^s_0(\Omega)\backslash\{0\}\::\: J'(w)(w)=\mathcal E(w,w)+ \lambda|w|_2^2-|w|_{p+1}^{p+1}=0\}.
\end{align}

\end{definition}

\begin{proposition}\label{proposition:leastEnergyOmegaProperties}
Let $\Omega\subset\mathbb R^N$, $N\geq 1$ be a smooth  bounded domain, $s\in (0,1)$, $p\in (1,2^\star_s-1)$, and $\lambda>-\lambda_1^s(\Omega)$. 
Then~\eqref{P} has a least-energy solution  $u$.  Furthermore $u>0$ in $\Omega$,
\begin{align}\label{ceps:eq}
c_\Omega=\inf_{w\in\cH^s_0(\Omega)\backslash \{0\}}\sup_{t\geq 0} J(tw)>0,
\end{align}
and
\begin{equation}
\label{MorseofLeastEnergy}
m(u)=1.
\end{equation}
\end{proposition}
\begin{proof}
This result is well known and can be shown with straightforward adaptations of the arguments used in the local case $s=1$.  Here we only recall the proof of~\eqref{MorseofLeastEnergy}. From~\eqref{firstEigenNegative}, we know that $m(u)\geq 1$. Next, we show the opposite inequality. To simplify the notation, let $\mathcal N:=\mathcal N(\Omega)$.
 Since a least-energy solution $u$ is a minimum of $J$ on the Nehari  manifold $\mathcal N$, we have that
 \begin{align}
 \mathcal B(v,v)=J''(u)(v,v)\geq 0\quad\mbox{ for all }v\in T_u\mathcal N,
 \end{align} 
 where $T_u\mathcal N$ denotes the tangent space to $\mathcal N$ at $u$. Furthermore, $\mathcal N$ is a $C^1$-Hilbert manifold of codimension $1$, indeed (see, for instance, \cite[Chapter 6.3]{A07})
 $\mathcal N$ is the zero level set of the $C^1$ functional $G:
 \cH_0^{s}(\Omega)\rightarrow \mathbb R$
 \[G(v):=J'(v)(v),\]
 and, for any $u\in\mathcal N$, we have that $G'(u)\neq 0$, since
 \[G'(u)(u):=J''(u)(u,u)+J'(u)(u)=J''(u)(u,u)=\mathcal B(u,u)=\|u\|^2-p|u|_{p+1}^{p+1}=(1-p)\|u\|^2<0.\]
 As a consequence, by Lemma~\ref{lemma:equivalenceMorse}, $m(u)\leq 1$.
 
 \begin{remark}[Symmetry and monotonicity of non-negative solutions]\label{remark:symmmetryU}
    Similarly to the observation in Remark~\ref{remark:symmmetryPhi1},   it is also well known that, when the domain $\Omega$  is convex in the $x_1$-direction and symmetric with respect to the hyperplane $H:=\left\lbrace x\in \mathbb R^N:\, x_1=0\right\rbrace$, then 
    any non-negative solution $u$ of~\eqref{P} has reflectional symmetry with respect to $H$ and is strictly monotone decreasing in $|x_1|$ for $x=(x_1,\ldots,x_N)\in\Omega$ (see  \cite[Corollary 1.2]{JW14}). In particular this holds true for the least-energy solutions $u$ of~\eqref{P}, since they are non-negative.
 \end{remark}
\end{proof}

\subsection{Ground-state solutions in $\mathbb R^N$}\label{section:groundStateRN}
Here, we consider the problem
\begin{align}\label{PQlambda}
    (-\Delta)^s Q + \lambda Q = Q^p\quad \text{ in }\R^N,\qquad Q\in H^s(\R^N),\qquad Q>0,
\end{align}
where $\lambda>0.$  For a solution $Q\in H^s(\R^N)$ of~\eqref{PQlambda} we let $L_+$ denote the corresponding linearized operator given by
\begin{align}\label{defLinearizedOperator}
    L_+:=(-\Delta)^s+\lambda-p|Q|^{p-1}
\end{align}
acting on $L^2(\mathbb R^N)$. 
\begin{definition}
    We say that a nontrivial solution $Q\geq 0$ of~\eqref{PQlambda}
is a ground state if it has Morse index 1, namely, if the linearized operator $L_+$ in $Q$ defined in~\eqref{defLinearizedOperator} has exactly one strictly negative eigenvalue (counted with multiplicity). 
\end{definition} 
\begin{proposition}[Existence]\label{prop:existence}
Let $N\geq1$, $s\in (0,1)$, $p\in (1,2^\star_s-1)$, and $\lambda>0$. Then~\eqref{PQlambda} has a ground state solution $Q$. Furthermore 
\begin{align}\label{c:eq}
c:=J(Q)=\inf_{\cN}J,
\end{align}
where the definition of $J$ is given in~\eqref{def:functionalJ} and $\mathcal N$ is the Nehari manifold 
\begin{align}\label{NehariRN}
\cN:=\{w\in H^s(\R^N)\backslash\{0\}\::\: J'(w)(w)=\|w\|^2-|w|_{p+1}^{p+1}=0\}.
\end{align}
\end{proposition}
The proof follows by adjusting the arguments in the local case $s=1$ with only minor changes using concentration-compactness in the fractional framework, see \cite{l84,palatucci}. 
 We also refer to \cite{Felmer_Quaas_Tan_2012, DipierroPalatucciValdinoci_2013, FLS16, Frank_Lenzmann_2013, S13} for related results.

\begin{proposition}[Uniqueness and nondegeneracy]\label{prop:nongeneracyLimitProblem}
Let $\lambda>0$ and $Q\geq 0$ be a ground state solution of~\eqref{PQlambda}.
Then it is unique (up to translations), strictly positive,  radially symmetric, and strictly decreasing in the radial variable. Furthermore, $Q$ is nondegenerate, in the sense that
\begin{align}\label{kernelLimitProblem}
    \operatorname{ker}(L_+)=\operatorname{span}(\partial_{x_1}Q,\ldots, \partial_{x_N}Q).
\end{align}
\end{proposition}
\begin{proof}
The case $\lambda=1$ has been proved in \cite{Frank_Lenzmann_2013} (for $N=1$) and \cite[Theorems 3 and 4]{FLS16} (for $N\geq 2$), the case $\lambda>0$,  $\lambda\neq 1$ then follows by a scaling argument.
\end{proof}

\section{A symmetry result for the linearized problem}\label{sec:symm}
Let us consider the hyperplane 
\[H:=\{x\in\mathbb R^N\ :\ x_1=0\}.\]
Let $\sigma: \mathbb R^N\to\mathbb R^N$ denote the reflection with respect to the hyperplane $H$, \emph{i.e.}
\begin{equation}\label{def:sigma}\sigma (x)=(-x_1,x_2,\ldots,x_N) \quad\text{ for any }\quad x=(x_1,x_2,\ldots,x_N)\in\mathbb R^N.\end{equation}

\begin{definition}
A function $w:\mathbb R^N\to\mathbb R$ is said to be: \begin{itemize}
    \item symmetric with respect to the hyperplane $H$ if
\begin{align*}
w(\sigma(x))=w(x)\quad  \text{ for } x\in \mathbb R^N.  
\end{align*}
\item
antisymmetric with respect to the hyperplane  $H$ if
\begin{align*}
w(\sigma(x))=-w(x)\quad  \text{ for } x\in \mathbb R^N.
\end{align*}
\end{itemize}
\end{definition}

In this Section we prove Theorem~\ref{sym}, namely that if the domain $\Omega$  is convex in the $x_1$-direction and symmetric with respect to $H$ then  any second eigenfunction $v$ of the  linearized eigenvalue problem at $u$:
\begin{equation}\label{LinearizzatoInSec3}
(-\Delta)^{s} v+\lambda v-pu^{p-1}v=\mu_2 v \quad \text{ in }\Omega,\qquad v=0\quad  \text{ in }\mathbb R^N\backslash\Omega.
\end{equation}
necessarily inherits the symmetry of the domain, when the second eigenvalue $\mu=\mu_2\leq 0$. 
\\

It is well know that any non-negative solution $u$ of~\eqref{P} and the first eigenfunction $\phi_1$  of the linearized eigenvalue problem~\eqref{eigenvalueProblematu} at $u$ are symmetric with respect to $H$ (see Remarks~\ref{remark:symmmetryPhi1} and~\ref{remark:symmmetryU}).
We stress that, unlike $u$ and $\phi_1$, second eigenfunctions  $v$  are sign-changing (see Lemma~\ref{eg vl neg}), hence different arguments must be developed in order to prove their symmetry.\\

Our proof  exploits nonlocal versions of the strong maximum principle, of the Hopf Lemma and of integration by parts formulas, following some ideas used in \cite{DamascelliGrossiPacellaAIHP1999} in the local case. Nevertheless since in the fractional framework the strong maximum principles and the Hopf Lemma 
 for sign-changing solutions is known to hold only when they are antisymmetric  (see \cite{JarohsWeth2016,FallJarohs2015}), we have to introduce new ideas.
More precisely, instead of working directly with a general second eigenfunction $v$, we first rearrange it by considering its polarization $Pv$ with respect to the hyperplane $H$.  Following some arguments similar as in \cite{Benedikt}, based on the introduction of a new variational characterization for the second eigenfunctions, we can show that also $Pv$ solves~\eqref{eq:secondEigen}.
Finally we obtain an antisymmetric second eigenfunction by taking the difference between $Pv$ evaluated at a point and at the point obtained by reflection with respect to the hyperplane $H$. This function turns out to have all the properties needed to conclude the proof.\\

The section is organized as follows:
\begin{itemize}
\item In Section~\ref{subsection:seconEigenfunction} we introduce the variational characterization for the second eigenfunction of the eigenvalue problem~\eqref{eigenvalueProblematu}, extending to our linearized problem the characterization introduced in  \cite{Benedikt} for the second eigenfunction of $(-\Delta)^s$.
\item In Section~\ref{subsection:Polarization} we recall the definition of the polarization of a function with respect to the hyperplane $H$. Then, when $\Omega$ is symmetric with respect to $H$ and convex  in the $x_1$-direction, we prove that  the polarization of a second eigenfunction is a second eigenfunction,
this is contained in Proposition~\ref{prop:polarizationIsEigenfunction}.
\item Finally, in Section~\ref{subsection:proofOfSymmetryResult} we conclude the proof Theorem~\ref{sym}.
\end{itemize}

\medskip

As a consequence of Theorem~\ref{sym}, we have the following result in the symmetric domains which satisfy the assumption $(\mathcal S)$ stated in the Introduction.
\begin{corollary} \label{corollary: symmetryOfKernel_CylindricalDomains} Assume that $\Omega$ satisfies $(\mathcal S)$ and let $u$ be a non-negative solution of~\eqref{Peps}.
If $\mu_2\leq 0$, then any nontrivial solution $v$ to~\eqref{LinearizzatoInSec3} is (sign-changing and) $G$-invariant.
\end{corollary}
In particular, in the case when $\Omega$ is a ball, we deduce Corollary \ref{corollary:radialKernel} in the Introduction.
Notice that the last sentences in Corollary~\ref{corollary:radialKernel}, concerning the number of changes of sign of $v$, is a consequence of Lemma~\ref{lemma:SignAlongBoundary}. \\

Corollaries~\ref{corollary:radialKernel} and~\ref{corollary: symmetryOfKernel_CylindricalDomains} are the main tools to prove the nondegeneracy  for least-energy solutions, in the case when $\Omega$ is a large ball or a more general large symmetric domain satisfying assumption $(\mathcal S)$, respectively (see Section~\ref{section:uniquenessLargeDomain}).

\subsection{Second eigenfunctions: a variational characterization }\label{subsection:seconEigenfunction}
In this section $\Omega\subset\mathbb R^N$, $N\geq 1$, is any smooth bounded domain. Let us observe that if $\phi_2$ denotes the second eigenfunction of the eigenvalue problem~\eqref{eigenvalueProblematu}, then $\pm\phi_2^\pm\in\mathcal H^s_0(\Omega)\setminus\{0\}$  and so it  can be used as test functions in the weak formulation of the eigenvalue problem with $\mu=\mu_2$, obtaining
\begin{equation}\label{equalitiesSecondaEigen}
\mu_2\int_{\Omega}(\phi_2^+)^2 \ dx=\mathcal B (\phi_2,\phi_2^+)\qquad\mbox{ and }\qquad \mu_2\int_{\Omega}(\phi_2^-)^2\ dx=-\mathcal B (\phi_2,\phi_2^-).
\end{equation}
Inspired by the ideas in \cite{Benedikt}, we introduce the following characterization for the second eigenfunction of the eigenvalue problem~\eqref{eigenvalueProblematu}.
\begin{lemma}\label{lemma:analog2.1}
Let $v\in \cH^s_0(\Omega)$ be such that $v^{\pm}\not\equiv 0$ in $\Omega$,
\begin{equation}\label{2Inequal}
\mu_2\int_{\Omega}(v^+)^2 \ dx\geq\mathcal B (v,v^+),
\end{equation}
and 
\begin{equation}\label{2InequalNegativepart}
\mu_2\int_{\Omega}(v^-)^2\ dx\geq-\mathcal B (v,v^-),
\end{equation}
where $\mathcal B$ is the bilinear form in~\eqref{BilinearForm} and $v^{\pm}$ are the positive and negative part of $v$, see~\eqref{defpositiveNegativaPart}.

Then $v $ is an eigenfunction associated to the eigenvalue $\mu_2$ for the linearized problem~\eqref{eigenvalueProblematu} at $u$ and equalities hold in~\eqref{2Inequal}.
\end{lemma}
\begin{proof} The proof is similar to the one of \cite[Lemma 2.1]{Benedikt}, we write its main arguments for completeness.
Let $\alpha_0>0$ be such that 
\begin{equation}
\int_{\Omega}\left( v^+-\alpha_0v^-\right) \phi_1\, dx=0,
\end{equation}
where $\phi_1>0$ in $\Omega$ is the first eigenfunction associated to the linearized problem~\eqref{eigenvalueProblematu} at $u$ with eigenvalue $\mu_1$
(such an $\alpha_0$ trivially exists since both $\int_{\Omega}v^{\pm}\phi_1>0$).

Since $v^{\pm}\in \mathcal H^s_0(\Omega)$, then 
$v^+-\alpha_0v^-\in\mathcal H^s_0(\Omega)\setminus\{0\}$ and so it can be used as a test function in the variational characterization~\eqref{variationalCharctmu2} for the second eigenvalue $\mu_2$, hence
\begin{align}\label{firstestmu2}
\mu_2&\leq\frac{\mathcal{B}(v^+-\alpha_0v^-,v^+-\alpha_0v^-)}{\int_{\Omega}\left(v^+-\alpha_0v^-\right)^2}=\frac{\mathcal{B}(v^+-\alpha_0v^-,v^+-\alpha_0v^-)}{\int_{\Omega}(v^+)^2dx+\alpha_0^2\int_{\Omega}(v^-)^2dx}.
\end{align}
On the other hand, by assumption~\eqref{2Inequal},
\begin{eqnarray*}
\mu_2\left[ \int_{\Omega}(v^+)^2dx+\alpha_0^2\int_{\Omega}(v^-)^2dx\right]\geq\mathcal B(v,v^+)-\alpha_0^2\mathcal B(v,v^-)=\mathcal B(v,v^+-\alpha_0^2v^-),
\end{eqnarray*}
namely,
\begin{eqnarray}\label{lowerbounmu2}
\mu_2\geq\frac{\mathcal B(v,v^+-\alpha_0^2v^-)}{ \int_{\Omega}(v^+)^2dx+\alpha_0^2\int_{\Omega}(v^-)^2dx
}.
\end{eqnarray}
Using the definition of $\mathcal B$, the fact that 
\begin{equation}\label{okioki}\mathcal E(v,v^+-\alpha_0^2v^-)\geq\mathcal E(v^+-\alpha_0v^-,v^+-\alpha_0v^-),
\end{equation} (we refer to the proof of \cite[Lemma 2.1]{Benedikt} for these computations) and observing, by simple computations, that
\begin{align}\label{oki}
\int_{\Omega}(\lambda-pu^{p-1})v(v^+-\alpha_0^2v^-)dx
&=\int_{\Omega}(\lambda-pu^{p-1})((v^+)^2+\alpha_0^2(v^-)^2)dx\\
&=\int_{\Omega}(\lambda-pu^{p-1})(v^+-\alpha_0v^-)^2dx,
\end{align}
it follows that
\begin{equation}\label{evalB}
\mathcal B(v,v^+-(\alpha_0)^2v^-)\geq\mathcal B(v^+-\alpha_0v^-,v^+-\alpha_0v^-).
\end{equation}
Combining~\eqref{firstestmu2},~\eqref{lowerbounmu2} and~\eqref{evalB},
\begin{eqnarray}\label{BETTERlowerbounmu2}
\mu_2\leq\frac{\mathcal B(v^+-\alpha_0v^-,v^+-\alpha_0v^-)}{ \int_{\Omega}(v^+)^2dx+\alpha_0^2\int_{\Omega}(v^-)^2dx
}\leq \frac{\mathcal B(v,v^+-\alpha_0^2v^-)}{ \int_{\Omega}(v^+)^2dx+\alpha_0^2\int_{\Omega}(v^-)^2dx
}\leq\mu_2.
\end{eqnarray}
Then, recalling~\eqref{variationalCharctmu2}, one can show that $v^+-\alpha_0v^-$ is a second eigenfunction.   Furthermore, it also follows  that
\[
\mathcal B(v,v^+-\alpha_0^2v^-)=\mathcal B(v^+-\alpha_0v^-,v^+-\alpha_0v^-)
\]
and so, by the definition of $\mathcal B$ and using that $v^+$ and $v^-$ have disjoint supports, we have that
\[
\mathcal E(v,v^+-\alpha_0^2v^-)=\mathcal E(v^+-\alpha_0v^-,v^+-\alpha_0v^-).
\]
Using this last equality and the bilinearity of $\mathcal E$, we obtain that $1+\alpha_0^2=2\alpha_0$ (as in the proof of \cite[Lemma 2.1]{Benedikt}). Then $\alpha_0=1$ and this concludes the proof.
\end{proof}

\subsection{Polarization of second eigenfunctions in symmetric domains}\label{subsection:Polarization}
Let $H$ be the hyperplane 
\begin{align}\label{Hdef}
H:=\{x\in\mathbb R^N\,:\, x_1=0\}    
\end{align}
and let us define the halfspaces
\[\Sigma^+:=\{x\in\mathbb R^N\ :\ x_1\geq 0\}\qquad\text{ and }\qquad  \Sigma^-:=\{x\in\mathbb R^N\ :\ x_1\leq 0\}.\]
For $w\in H^s(\mathbb R^N)$, we define the polarization of $w$ with respect to $H$ as the function \begin{equation}\label{def:polarization}(Pw)(x):=\left\{\begin{array}
 {lr}
 \min\{w(x),w(\sigma(x))\}, &\ x\in \Sigma^+,\\
 & \\
 \max\{w(x),w(\sigma(x))\},&\ x\in\Sigma^-,
 \end{array} \right.
 \end{equation}
 where $\sigma(x)\in\mathbb R^N$ is the reflection of $x\in\mathbb R^N$ with respect to the hyperplane $H$ (see~\eqref{def:sigma}).
 Let us recall that
\begin{equation}\label{normsPreservingPolarization}
[(Pw)^{\pm}]_{H^s(\mathbb R^N)} \leq [w^{\pm}]_{H^s(\mathbb R^N)}\qquad\mbox{ and }\qquad   |(Pw)^{\pm}|_{2} =|w^{\pm}|_{2},
\end{equation}
where $[v]_{H^s(\mathbb R^N)}^2:=\mathcal E(v,v)$ (see, for instance, \cite[Lemma 2.3]{Benedikt}).

\begin{lemma}\label{lemma:analog2.3_VECCHIO}
Let $N\geq 1$ and assume that $\Omega\subset\mathbb R^N$ is a smooth bounded domain which is convex in the $x_1$-direction and symmetric with respect to the hyperplane $H$ (given by \eqref{Hdef}). Let $u$ be a non-negative solution of~\eqref{Peps} and let $\mathcal B$ be the associated bilinear form as defined in~\eqref{BilinearForm}.

Then $Pw\in \mathcal H^s_0(\Omega)$ for any $w\in \mathcal H^s_0(\Omega)$; furthermore,
\begin{equation}\label{disPositivevecchio}
\mathcal B (w,w^+)\geq\mathcal B (Pw,(Pw)^+)
\end{equation}
and 
\begin{equation}\label{disNegativevecchio}
-\mathcal B (w,w^-)\geq -\mathcal B (Pw,(Pw)^-).
\end{equation}

\end{lemma}
\begin{proof}
Note that $Pw\in H^s(\mathbb R^N)$ by~\eqref{normsPreservingPolarization} and from \cite[Lemma 2.3]{Benedikt}, we know that 
\begin{eqnarray}\label{BenediktTerm}
\mathcal E (w,w^+)\geq\mathcal E (Pw,(Pw)^+)\quad \text{ and }\quad 
-\mathcal E (w,w^-)\geq -\mathcal E (Pw,(Pw)^-).
\end{eqnarray}
Moreover, since $w=0$ in  $\mathbb R^N\setminus \Omega$ and $\Omega$ is symmetric,  it follows by the definition of the polarization that $Pw=0$ in  $\mathbb R^N\setminus \Omega$, hence $Pw\in \mathcal H^s_0(\Omega)$.
\\
Using~\eqref{normsPreservingPolarization},
\begin{eqnarray}\label{lambdaone}
\lambda\int_{\Omega}ww^+=\lambda|w^+|^2_2=\lambda|(Pw)^+|^2_2=\lambda\int_{\Omega}(Pw)(Pw)^+
\end{eqnarray}
and similarly
\begin{eqnarray}\label{lambdatwo}
-\lambda\int_{\Omega}ww^-=\lambda|w^-|^2_2=\lambda|(Pw)^-|^2_2 = -\lambda\int_{\Omega} (Pw)(Pw)^-.
\end{eqnarray}
Next, exploiting the symmetry of $u$ with respect to $H$, we show that 
\begin{eqnarray}
\label{toDo1vecchio}
\int_{\Omega} u^{p-1}(w^{\pm})^2\,dx=\int_{\Omega} u^{p-1}\left((Pw)^{\pm}\right)^2\,dx.
\end{eqnarray}
Let us write $\Omega$ as $\Omega=A\cup (\Omega\setminus A),$ where
\begin{align}\label{A:def}
A:=\{x\in \Omega\ :\ (Pw)(x)=w(x)\}.    
\end{align}
Then, 
\begin{equation}\label{partevecchio}
\int_{A} u^{p-1}\left((Pw)^{\pm}\right)^2\,dx=\int_{A} u^{p-1}(w^{\pm})^2\,dx.
\end{equation}
Let 
\[
\Omega\setminus A=\{x\in \Omega\ : \ (Pw)(x)=w(\sigma(x))\}=K^+\cup K^-,
\]
where 
\[K^+:=(\Omega\setminus A)\cap\Sigma^+=\{x\in \Omega\ : \ (Pw)(x)=w(\sigma(x))<w(x)\} \ \subset \Sigma^+,\]
\[K^-:=(\Omega\setminus A)\cap\Sigma^-=\{x\in \Omega\ : \ (Pw)(x)=w(\sigma(x))>w(x)\}\ \subset \Sigma^-.\]
 By definition $K^+\cap K^-=\emptyset$,  and 
\[\sigma(K^{\pm})=K^{\mp},\]
so $\Omega\setminus A$  is invariant by the reflection with respect to the hyperplane $H$.
In order to prove~\eqref{toDo1vecchio} we split the integral and make a change of variable, that is,
\begin{eqnarray}\nonumber&&\int_{\Omega\setminus A}\left[u^{p-1}((Pw)^{\pm})^2-u^{p-1}(w^{\pm})^2\right]dx=\\\nonumber
&&=\int_{K^-}u^{p-1}(x)\left[((Pw)^{\pm})^2(x)-(w^{\pm})^2(x)\right]dx
+\int_{K^+}u^{p-1}(x)\left[((Pw)^{\pm})^2(x)-(w^{\pm})^2(x)\right]dx
\\\nonumber
&&=
\int_{K^-}u^{p-1}(x)\left[(w^{\pm})^2(\sigma(x))-(w^{\pm})^2(x)\right]dx
+\int_{K^+}u^{p-1}(x)\left[(w^{\pm})^2(\sigma(x))-(w^{\pm})^2(x)\right]dx
\\\nonumber
&&=
\int_{K^-}u^{p-1}(x)\left[(w^{\pm})^2(\sigma(x))-(w^{\pm})^2(x)\right]dx
+\int_{K^-}u^{p-1}(\sigma(x))\left[(w^{\pm})^2( x)-(w^{\pm})^2(\sigma(x))\right]dx
\\\label{pezzovecchio}
&&=\int_{K^-} \left[(w^{\pm})^2(\sigma(x))-(w^{\pm})^2(x)\right]\left[u^{p-1}(x)-u^{p-1}(\sigma(x))\right]= 0,
\end{eqnarray}
where the equality follows by the symmetry of $u$ with respect to $H$ (see Remark~\ref{remark:symmmetryU}).   Then 
\eqref{toDo1vecchio} follows from~\eqref{partevecchio} and~\eqref{pezzovecchio}.\\
Finally~\eqref{disPositivevecchio} and~\eqref{disNegativevecchio} follow by the definition of the bilinear form $\mathcal B$ in~\eqref{BilinearForm} and by ~\eqref{toDo1vecchio},~\eqref{lambdaone}-\eqref{lambdatwo} and~\eqref{BenediktTerm}.
\end{proof}

We are now ready to prove the main result of this subsection. 

\begin{proposition}\label{prop:polarizationIsEigenfunction}
Let $N\geq 1$ and assume that $\Omega\subset\mathbb R^N$ is a smooth bounded domain which is convex in the $x_1$-direction and symmetric with respect to the hyperplane $H$. Let $u$ be a non-negative solution of~\eqref{Peps}.
If $v$ is a second eigenfunction for  the linearized eigenvalue problem~\eqref{eigenvalueProblematu} at $u$, then the polarization $Pv$ of $v$ is a second eigenfunction for the same eigenvalue problem.
\end{proposition}

\begin{proof} 
Since $v$ is a second eigenfunction then, $v^{\pm}\not\equiv 0$ and, by~\eqref{equalitiesSecondaEigen}, $v^+$ and $v^-$ satisfy the equalities
\[\mu_2\int_{\Omega}(v^+)^2 \ dx=\mathcal B (v,v^+)\qquad\mbox{ and }\qquad \mu_2\int_{\Omega}(v^-)^2\ dx=-\mathcal B (v,v^-).
\]
Furthermore $(Pv)^{\pm}\not\equiv 0$ and, by Lemma~\ref{lemma:analog2.3_VECCHIO}, $Pv \in \mathcal H^s_0(\Omega)$ and
\[
\mathcal B (v,v^+)\geq\mathcal B (Pv,(Pv)^+)
\]
and 
\[
-\mathcal B (v,v^-)\geq -\mathcal B (Pv,(Pv)^-).
\]
As a consequence
\[\mu_2\int_{\Omega}(v^+)^2 \ dx\geq\mathcal B (Pv,(Pv)^+)\quad\mbox{ and }\quad \mu_2\int_{\Omega}(v^-)^2\ dx\geq-\mathcal B (Pv,(Pv)^-).
\]
Since (see by~\eqref{normsPreservingPolarization})
\[\mu_2\int_{\Omega}((Pv)^{\pm})^2=\mu_2\int_{\Omega}(v^{\pm})^2,\]
then
\begin{equation}\mu_2\int_{B}((Pv)^+)^2 \ dx\geq\mathcal B (Pv,(Pv)^+)\quad\mbox{ and }\quad \mu_2\int_{B}((Pv)^-)^2\ dx\geq-\mathcal B (Pv,(Pv)^-).
\end{equation}
The conclusion follows by Lemma~\ref{lemma:analog2.1}. 
\end{proof}

\subsection{The proof of Theorem~\ref{sym}}\label{subsection:proofOfSymmetryResult}
We prove  now the main result of this section.
\begin{proof}[Proof of Theorem~\ref{sym}]
Let $u$ be a solution to problem~\eqref{Peps}
 and let $v$ be a second eigenfunction for the eigenvalue problem~\eqref{LinearizzatoInSec3}. Let $Pv$ be the polarization of $v$ as defined in~\eqref{def:polarization}. Proposition~\ref{prop:polarizationIsEigenfunction} implies that $Pv$ is a second eigenfunction for the eigenvalue problem~\eqref{LinearizzatoInSec3}. Let us define
\[w(x):=(Pv)(\sigma (x))-(Pv)(x).\]
\emph{Step 1. We show that $w$ is an antisymmetric solution to~\eqref{LinearizzatoInSec3}}.

Indeed, for all $\varphi \in \cH_0^{s}(\Omega)$,
\begin{eqnarray} \label{soluzione1}
&&  
\int_{\mathbb R^N}\int_{\mathbb R^N}\frac{ \big(w(x)-w(y)\big)		\big(\varphi(x)-\varphi(y)\big)}{|x-y|^{N+2s}}\,dx\,dy=
\nonumber\\
&&  
=\int_{\mathbb R^N}\int_{\mathbb R^N}\frac{ \Big((Pv)(\sigma (x))-(Pv)(x)-(Pv)(\sigma (y))+(Pv)(y)\Big)		\big(\varphi(x)-\varphi(y)\big)}{|x-y|^{N+2s}}\,dx\,dy
\nonumber\\
&&  =-\int_{\mathbb R^N}\int_{\mathbb R^N}\frac{ \big((Pv)(x)-(Pv)(y)\big)		\big(\varphi(x)-\varphi(y)\big)}{|x-y|^{N+2s}}\,dx\,dy+\nonumber\\ 
&&\qquad +
\int_{\mathbb R^N}\int_{\mathbb R^N}\frac{ \big((Pv)(\sigma (x))-(Pv)(\sigma (y))\big)		\big(\varphi(x)-\varphi(y)\big)}{|x-y|^{N+2s}}\,dx\,dy
\end{eqnarray}
where, by a change of variables, 
\begin{eqnarray}\label{soluzione2}
&&\int_{\mathbb R^N}\int_{\mathbb R^N}\frac{ \big((Pv)(\sigma (x))-(Pv)(\sigma (y))\big)		\big(\varphi(x)-\varphi(y)\big)}{|x-y|^{N+2s}}\,dx\,dy=
\nonumber\\
&&
= \int_{\mathbb R^N}\int_{\mathbb R^N}\frac{ \big((Pv)(x)-(Pv)(y)\big)		\big(\varphi(\sigma (x))-\varphi(\sigma (y))\big)}{|\sigma (x)-\sigma (y)|^{N+2s}}\,dx\,dy
\nonumber \\
&&
=\int_{\mathbb R^N}\int_{\mathbb R^N}\frac{ \big((Pv)(x)-(Pv)(y)\big)		\big(\varphi(\sigma (x))-\varphi(\sigma (y))\big)}{|x-y|^{N+2s}}\,dx\,dy.
\end{eqnarray}
Hence, substituting~\eqref{soluzione2} into~\eqref{soluzione1} and using that $Pv$ is a solution of the eigenvalue problem~\eqref{LinearizzatoInSec3}, 
\begin{eqnarray} \label{soluzione3}
&&  
\int_{\mathbb R^N}\int_{\mathbb R^N}\frac{ \big(w(x)-w(y)\big)		\big(\varphi(x)-\varphi(y)\big)}{|x-y|^{N+2s}}\,dx\,dy=
\nonumber\\
&&  =-\int_{\mathbb R^N}\int_{\mathbb R^N}\frac{ \big((Pv)(x)-(Pv)(y)\big)		\big(\varphi(x)-\varphi(y)\big)}{|x-y|^{N+2s}}\,dx\,dy +\nonumber\\
&&\qquad +  
\int_{\mathbb R^N}\int_{\mathbb R^N}\frac{ \big((Pv)(x)-(Pv)(y)\big)		\big(\varphi(\sigma (x))-\varphi(\sigma (y))\big)}{|x-y|^{N+2s}}\,dx\,dy
\nonumber
\\
&&=  - \int_\Omega \left( -\lambda +\mu_2+pu(x)^{p-1}\right)(Pv)(x)\varphi(x)\,dx+\nonumber\\
&&\qquad+\int_\Omega \left( -\lambda +\mu_2 +pu(x)^{p-1}\right)(Pv)(x)\varphi(\sigma(x))\,dx.
\end{eqnarray}
Furthermore, by a change of variables, exploiting the symmetry of $\Omega$ and of the solution $u$ (\emph{i.e.} $\sigma (\Omega)=\Omega$ by assumption and $u(x)=u(\sigma (x))$, see Remark~\ref{remark:symmmetryU}) one obtains that
\begin{eqnarray}\label{soluzione4}
&&\int_\Omega \left( -\lambda +\mu_2+pu(x)^{p-1}\right)(Pv)(x)\varphi(\sigma(x))\,dx=\nonumber\\
&&\qquad =\int_{\sigma(\Omega)} \left( -\lambda +\mu_2+pu(\sigma (x))^{p-1}\right)(Pv)(\sigma (x))\varphi(x)\,dx
\nonumber\\
&&\qquad = \int_{\Omega} \left( -\lambda +\mu_2+pu(x)^{p-1}\right)(Pv)(\sigma (x))\varphi(x)\,dx.
\end{eqnarray}
Hence, substituting~\eqref{soluzione4} into~\eqref{soluzione3}, we have that $w$ is a weak solution of~\eqref{LinearizzatoInSec3}, because
\begin{eqnarray}\label{soluzione5}
&&  
\int_{\mathbb R^N}\int_{\mathbb R^N}\frac{ \big(w(x)-w(y)\big)		\big(\varphi(x)-\varphi(y)\big)}{|x-y|^{N+2s}}\,dx\,dy=
\nonumber\\
&&  = -\int_\Omega \left( -\lambda +\mu_2+pu(x)^{p-1}\right)(Pv)(x)\varphi(x)\,dx+\int_{\Omega} \left( -\lambda +\mu_2+pu(x)^{p-1}\right)(Pv)(\sigma (x))\varphi(x)\,dx
\nonumber
\\
&& =\int_{\Omega} \left( -\lambda +\mu_2+pu(x)^{p-1}\right)w(x)\varphi(x)\,dx.
\end{eqnarray} The antisymmetry of $w$ is trivial.
 This concludes the proof of \emph{Step 1.}

\medskip
 
\emph{Step 2. We show that
\begin{equation}\label{wNotNodalinH}
 w\geq 0 \text{ in } \Sigma^+.
\end{equation}}

\medskip

If $x\in\Sigma^+$ then $\sigma (x)\in \Sigma^- $, so, by the  definition of the polarization, it follows that
\[(Pv)(x)=\min\{v(x),v(\sigma(x))\}, \]
while
\[(Pv)(\sigma(x))=\max\{v(\sigma(x), v(\sigma(\sigma(x))=v(x)\}.\]
As a consequence,
\[(Pv)(\sigma(x))\geq (Pv)(x)\]
and the conclusion follows by the definition of $w$.

\medskip

\emph{Step 3. We show that $w\equiv 0$ (namely that $Pv$ is symmetric with respect to $H$).}

\medskip

By \emph{Step 2.}, we can use the nonlocal antisymmetric maximum principle 
(see \cite{JarohsWeth2016}), and deduce that
\begin{equation}\label{altern}
	\text{either } w\equiv 0 \text{ or } w>0 \text{ in } K \text{ for any compact } K  \text{ of } \Sigma^+\cap\Omega.
\end{equation}
We claim that the second alternative cannot hold.  To obtain the claim we have taken inspiration from the proof of \cite[Lemma 2]{lin1994}, related to a local problem.
Assume, by contradiction, that the second alternative holds.  Then,
by  the  nonlocal Hopf lemma for antisymmetric supersolutions in \cite[Proposition 3.3]{FallJarohs2015}, we derive that 
\begin{equation}\label{signDerivativeW}
 \frac{w}{\delta^{s}}>0\mbox{ on  }\partial\Omega\cap \Sigma^+.
\end{equation}
Here $\delta(x):=\operatorname{dist}(x,\partial \Omega)$ and the quotient $\frac{w}{\delta^{s}}$ is understood at $\partial \Omega$ as a limit from inside $\Omega$.

Next, using the integration by parts formula in \cite[Theorem 1.9]{Ros-OtonSerra2014} (see also \cite{RosOtonSurvey}) applied to $u$ and $w$, we obtain that
\begin{equation}\label{byparts}
\int_{\Omega}\frac{\partial u}{\partial x_{1}}(-\Delta)^{s}w + \frac{\partial w}{\partial x_{1}}(-\Delta)^{s}u\,dx=
-\Gamma (1+s)^2\int_{\partial\Omega}\frac{u}{\delta^{s}}\frac{w}{\delta^{s}}\nu_{1}\,d\sigma,\end{equation}
where $\nu_{1}$ is the first component of $\nu$ 
 the  unit outward normal to $\partial\Omega$ at $x$.
 
Observe that, since $u$ and $w$ solve problems~\eqref{Peps} and~\eqref{LinearizzatoInSec3} respectively, one has that
\begin{align*}
\int_{\Omega}&
\frac{\partial u}{\partial x_{1}}(-\Delta)^{s}w + \frac{\partial w}{\partial x_{1}}(-\Delta)^{s}u\,dx\\
&= 
\int_{\Omega} \frac{\partial u}{\partial x_{1}} p u^{p-1}w\,dx+\int_{\Omega} \frac{\partial w}{\partial x_{1}}u^p\,dx
-\lambda \int_\Omega   \frac{\partial u}{\partial x_{1}}w\,dx -\lambda \int_\Omega  \frac{\partial w}{\partial x_{1}}u\,dx+\mu_2\int_{\Omega} \frac{\partial u}{\partial x_{1}} w\,dx\\
 &=\int_\Omega \frac{\partial (u^p)}{\partial x_{1}}w\,dx+\int_{\Omega} \frac{\partial w}{\partial x_{1}}u^p\,dx +\mu_2\int_{\Omega} \frac{\partial u}{\partial x_{1}} w\,dx= \mu_2\int_{\Omega} \frac{\partial u}{\partial x_{1}} w\,dx.
\end{align*}
Hence,~\eqref{byparts} implies that
\begin{equation}\label{Kidentity}
\Gamma (1+s)^2\int_{\partial\Omega}\frac{u}{\delta^{s}}\frac{w}{\delta^{s}}\nu_{1}\,d\sigma=-\mu_2\int_{\Omega} \frac{\partial u}{\partial x_{1}} w\,dx \leq 0.
\end{equation} 
Where the last inequality in~\eqref{Kidentity} follows from the fact that by assumption $\mu_2\leq 0$, and moreover 
\[\frac{\partial u}{\partial x_{1}} w<0 \quad\mbox{ in }\Omega.\]
Indeed  $w>0$ in $\Omega\cap\Sigma^+$ and so, by the antisymmetry,  $w<0$ in $\Omega\cap\Sigma^-$. Moreover $u$ is radial and monotone decreasing in $|x_1|$, so $\frac{\partial u}{\partial x_{1}} <0$ in  $\Omega\cap\Sigma^+$ and $\frac{\partial u}{\partial x_{1}} >0$ in  $\Omega\cap\Sigma^-$.
\\
Finally let us observe that by the Hopf boundary Lemma for the nonlocal case  (recall that $u$ is positive) we know that 
\[
\frac{u}{\delta^{s}}>0\quad \mbox{ on }\partial\Omega.
\]
Moreover, from~\eqref{signDerivativeW}, the antisymmetry of $w$ and the geometric assumptions on $\Omega$ we deduce that
\[
\frac{w}{\delta^{s}}\nu_{1}\geq 0
\quad \text{ and }\quad 
\frac{w}{\delta^{s}}\nu_{1}\not\equiv 0\quad \mbox{ on }\partial \Omega.
\]
Hence,~\eqref{Kidentity} gives a contradiction. As a consequence, recalling~\eqref{altern}, we have that $w\equiv 0$. 

\medskip

\emph{Step 4. We show that  $v$ is symmetric with respect to $H$.}
\medskip
Let $x\in \Sigma^+$, then $\sigma(x)\in \Sigma^-$; hence, using the definition of the polarization, we have
\[(Pv)(x)=\min\{v(x), v(\sigma(x))\},\]
\[(Pv)(\sigma( x))=\max\{v(\sigma( x)), v(\sigma(\sigma( x)))=v(x)\}.\]
By \emph{Step 3,} $(Pv)$ is symmetric w.r.t. the hyperplane $H$ so, 
$(Pv)(x)=(Pv)(\sigma( x))$, hence
\[\min\{v(x), v(\sigma(x))\}=\max\{v(x), v(\sigma(x))\},\]
namely $v(x)=v(\sigma( x))$. For $x\in\Sigma^-$ we can repeat a similar argument,
and the proof is complete.
\end{proof}

\begin{remark}\label{rkk:polarization_only_with_a=0}
If one were able to exhibit the existence of a non-symmetric second eigenfunction of \eqref{linearizedEqINTRO}, then Theorem~\ref{sym} would imply that $\mu_2>0$ necessarily. Since we know that $\mu_1<0$, this would immediately imply the nondegeneracy for \emph{any} non-negative solution of~\eqref{Peps} and, from it, the uniqueness would also follow using the uniqueness and nondegeneracy in the asymptotically linear regime (i.e. as $p\to 1^+$), which has been
recently proved in the nonlocal framework in \cite{DIS22}.

In \cite{Benedikt}, the nonradiality of the second eigenfunction of the fractional Laplacian $(-\Delta)^s$ in the ball is shown by considering a polarization of a radial second eigenfunction with respect to the hyperplane $H_a=\{x\in\mathbb R^N: x_1=a\}$ for a suitable $a\neq 0$. Unfortunately, due to the presence of the non-autonomous potential $pu^{p-1}$ in our linearized equation (see \eqref{linearizedEqINTRO}), this idea cannot be used in our setting.  Indeed, let $w$ be a second radial eigenfunction of \eqref{linearizedEqINTRO}, then, using Lemma~\ref{lemma:SignAlongBoundary}, we may assume, without loss of generality, that $w\geq 0$ in an annular region about the boundary (otherwise we take $-w$); hence, it is not difficult to see that $(P_aw)^\pm\in \mathcal H^s_0(B)$ for $a>0$ small.  Then, similar computations as those in the proof of Lemma~\ref{lemma:analog2.3_VECCHIO}, exploiting the symmetry and monotonicity of the term $pu^{p-1}$, show that the $2$ equalities in~\eqref{toDo1vecchio} reduce, if $a>0$, to the following \emph{strict} inequalities
    \begin{eqnarray}
&\int_{B} u^{p-1}(w^{+})^2\,dx<\int_{B} u^{p-1}\left((P_aw)^{+}\right)^2\,dx,\nonumber
\\
\label{BIGPROBLEM}
&\int_{B} u^{p-1}(w^{-})^2\,dx>\int_{B} u^{p-1}\left((P_aw)^{-}\right)^2\,dx
\end{eqnarray}
(note that, in this case, $B\setminus A\neq \emptyset$ with $A$ as in \eqref{A:def}, since $w$ is radial and $a>0$).
In particular, \eqref{BIGPROBLEM} does not allow now to conclude that ~\eqref{disNegativevecchio} in Lemma~\ref{lemma:analog2.3_VECCHIO} holds true (one needs the \emph{opposite} inequality in~\eqref{BIGPROBLEM}) and this, in turn, prevents the use of Lemma \ref{lemma:analog2.1}. Therefore, a contradiction with Theorem \ref{sym} cannot be achieved in this way.
\end{remark}

\section{Uniqueness and nondegeneracy of least-energy solutions in large symmetric domains}\label{section:uniquenessLargeDomain}

This section is devoted to the proof of  Theorem~\ref{theorem:uniquenessLARGE}.
Theorem~\ref{theorem:uniquenessAndNondegeneracyB} is a special case of Theorem~\ref{theorem:uniquenessLARGE}.

In what follows we assume that $\lambda >0,$ and we recall that
\begin{align*}
\|w\|^2:= \lambda|w|_2^2+\cE(w,w)
\end{align*}
as given in~\eqref{norm:def}.  For $m\in \N$, let $O(m)$ denote the group of linear isometries of $\R^m$.   Let $N\geq 2$, $\ell\geq 1,$ and $m_i\geq 2$ for $i=1,\ldots,\ell$ be such that $\sum_{i=1}^\ell m_i =N$. We set
\begin{align}\label{Gsection}
  G:=O(m_1)\times \ldots \times O(m_\ell)\subseteq O(N).  
\end{align}
\begin{definition} 
We say that a domain $D\subset\mathbb R^N$ is $G$-invariant if, for every $x\in D\subset \R^{m_1}\times \ldots \times \R^{m_\ell}=\R^N$, we have that 
\begin{align*}
gx = (g_1x_1,\ldots,g_\ell x_\ell)\in  D\quad \text{ for all }g=(g_1,\ldots,g_\ell)\in O(m_1)\times \ldots \times O(m_\ell)=G.    
\end{align*}
\end{definition}

\begin{definition} Let
\begin{align}\label{HsG:def}
H^s_G(\mathbb R^N):=\left\{w\in H^s(\mathbb R^N)\ :\ w(gx)=w(x)  \mbox{ for any } g\in G\right\}
\end{align}
denote the subspace of $G$-invariant functions in $H^s(\mathbb R^N)$. 
\end{definition}

These symmetries have infinite orbits. Therefore, as is well known, the space $H_G^s(\R^N)$ can be compactly embedded in $L^{p+1}(\R^N)$.  We state this result next. 

\begin{lemma}\label{prop:compact} Let $N\geq 2$ and let $G$ be as in \eqref{Gsection}.
Then the embedding $H^s_G(\R^N)\hookrightarrow L^{p+1}(\R^N)$ is compact, for all $p\in (1,2^\star_s-1)$.
\end{lemma}
\begin{proof}
Since $\mathcal O_x:=\{ gx \,  :\, g\in G\}$ the $G$-orbit of any $x\in\Omega\setminus\{0\}$  is infinite, the proof follows classical arguments that can be found, for instance, in \cite{CS19}. The extension to the fractional case is straightforward using the fractional version of the concentration compactness lemma \cite{S13}.
\end{proof}
\begin{remark}\label{rmk:compactembeddings}
We stress that Lemma \ref{prop:compact} holds under the assumption $N\geq 2$. For this reason  in Theorem \ref{theorem:uniquenessAndNondegeneracyB} and in Theorem \ref{theorem:uniquenessLARGE} we do not consider the case $N=1$.\\
Furthermore let us observe that if $G=O(N)$,  then Lemma \ref{prop:compact} gives the compactness of the embedding 
$H^s_{rad}(\R^N)\hookrightarrow L^{p+1}(\R^N)$.\\
\end{remark}
In this section we consider domains $D\subset \mathbb R^N$ that satisfy the following assumption.
\begin{itemize}
\item[$(\mathcal S)$] $D \subset\mathbb R^N$ is a Lipschitz bounded $G$-invariant domain and convex in the $x_i$-direction, 
for any $i=1,\ldots,N$.
\end{itemize}
Let us observe that, if $D$ satisfies $(\mathcal S)$, then
\begin{align*}
(x_1,\ldots,x_{i-1},tx_i,x_{i+1},\ldots,x_N)\in D\qquad \text{for every $x=(x_1,\ldots,x_N)\in  D,$ $t\in[-1,1],$}    
\end{align*}
 and $i=1,\ldots,N$.

In the following, $D$ denotes a smooth open bounded set satisfying $(\mathcal S)$ and, for $R>0$, let
\begin{align}\label{dr:def}
 D_R:=RD=\{ x \in \mathbb R^N\ :\ R^{-1}x\in D \}.
 \end{align}

Let $\Omega=D_R$, denote by $u_R$ a least-energy solution of~\eqref{Peps} and by $c_R:=c_{\Omega}$ its energy level, respectively (see the definitions in~\ref{def:leastEnergyBounded}).
Furthermore let $Q$ and $c$ be the ground state solution in $\mathbb R^N$ and its energy level, respectively (see~\eqref{PQlambda} and~\eqref{c:eq}).
\subsection{Convergence}
We first argue that a least-energy solution $u_R$ of~\eqref{Peps} converges to a ground-state solution $Q$ of~\eqref{PQlambda}.   Recall the definition of the Nehari manifold $\cN(\Omega)$ given in~\eqref{Nehari:def}.

\begin{lemma}\label{lem:bds}
There is $C=C(\lambda,D)>0$ such that $|c_R|<C$ and $|u_R|_{p+1}^{p+1}=\|u_R\|^2<C$ for all $R\geq 1.$
\end{lemma}
\begin{proof}
    Let $R\geq 1$ and let $\varphi\in C^\infty_c(\R^N)$ be such that $\operatorname{supp}\varphi\subset  D\subseteq  D_R$. Let
    \begin{align*}
        t=\left(\frac{\|\varphi\|^2}{|\varphi|_{p+1}^{p+1}}\right)^\frac{1}{p-1}.
    \end{align*} 
    Then $t\varphi\in \cN(D)\subset \cN(D_R)$ and
    \begin{align*}
    c_R  = J(u_R) = \inf_{w\in \cN(D_R)}J(w)\leq J(\varphi)=:C_1
    \end{align*}
for all $R\geq 1$.  Moreover, $|u_R|_{p+1}^{p+1}=\|u_R\|^2 = (\frac{1}{2}-\frac{1}{p+1})^{-1}c_R\leq (\frac{1}{2}-\frac{1}{p+1})^{-1}C_1=:C$.
\end{proof}

The next Lemma shows that the energy of the growing domains converges to the energy of the ground state in $\R^N.$

\begin{lemma}\label{lem:ceps}
$c_R \to c$ and $\|u_R\|\to\|Q\|^2$ as $R\to +\infty.$
\end{lemma}
\begin{proof}
By Lemma~\ref{lem:bds}, there is $c_*\geq 0$ such that $c_R\to c_*$ as $R\to +\infty$.  We claim that $c_*=c.$

Note that, since $\cH^s_0( D_R)\subset H^s(\R^N)$, then $c\leq c_R$ for all $R\to +\infty,$ which yields that $c\leq c_*$.  Next, let $Q\in H^s(\R^N)$ denote the unique radial least-energy solution of~\eqref{PQlambda}.  By density, there is a sequence $(\varphi_n)\subset C^\infty_c(\R^N)$ such that $\varphi_n\to Q$ in $H^s(\R^N)$ as $n\to\infty$.  By Sobolev embeddings, this also implies that $\varphi_n\to Q$ in $L^{p+1}(\R^N)$ as $n\to\infty$. Let $(R_n)\subset (0,+\infty)$ be such that $R_n\to +\infty$ as $n\to\infty$ and $\operatorname{supp}(\varphi_n)\subset  D_{R_n}$; in particular, $\varphi_n\in \cH^s_0(D_{R_n})$.  Note that 
\begin{align*}
t_n:=\left( \frac{\|\varphi_n\|^2}{|\varphi_n|_{p+1}^{p+1}} \right)^\frac{1}{p-1}
\to \left( \frac{\|Q\|^2}{|Q|_{p+1}^{p+1}} \right)^\frac{1}{p-1}=1\qquad \text{ as }n\to\infty.    
\end{align*}
Moreover, $\|t_n\varphi_n\|^2 = |t_n\varphi_n|_{p+1}^{p+1}$ for every $n\in \N$.  Then,
\begin{align*}
c_{R_n}\leq J(t_n\varphi_n)=t_n^2\left(\frac{1}{2}-\frac{1}{p+1}\right)\|\varphi_n\|^2\to \left(\frac{1}{2}-\frac{1}{p+1}\right)\|Q\|^2=c\qquad \text{ as }n\to\infty.
\end{align*}
This yields that $c_*\leq c$, and therefore $\lim_{n\to \infty}c_{R_n}=c_*=c$.  Since the limit is independent of the sequence $R_n$, we have that $\lim_{R\to +\infty}c_{R}=c$. Moreover,
\begin{align*}
 \|u_R\|^2 = \left(\frac{1}{2}-\frac{1}{p+1}\right)^{-1}c_R \to \left(\frac{1}{2}-\frac{1}{p+1}\right)^{-1}c =\|Q\|^2 \qquad \text{ as }R\to +\infty.
\end{align*}
This ends the proof.
\end{proof}

The following Proposition establishes the convergence of the least-energy solution in the growing domain to the ground state in $\R^N.$
\begin{proposition}\label{prop:utoQ}
$u_R\to Q$ in $H^s(\R^N)$ as $R\to +\infty$.
\end{proposition}
\begin{proof}
Let $(R_n)\subset (1,+\infty)$ be such that $\lim_{n\to\infty}R_n=+\infty$. By Lemma~\ref{lem:bds}, we know that $\|u_{R_n}\|^2<C$ for all $n\in \mathbb N.$ Then there is $u_*\in H^s(\R^N)$ such that (passing to a subsequence) 
\begin{align}\label{cm1}
u_{R_n} \weakto u_*\qquad \text{ weakly in $H^s(\R^N)$ as $n\to\infty$}
\end{align}
 and strongly in $L_{loc}^2(\R^N)$. Moreover, by Lemma~\ref{lem:ceps}, we know that 
\begin{align}\label{c0}
|u_{R_n}|_{p+1}^{p+1}=\|u_{R_n}\|^{2}\to \|Q\|^2=|Q|_{p+1}^{p+1}>0\qquad  \text{ as }n\to\infty.
\end{align}
Let $w_n:=u_{R_n}-u_*$. We claim that 
\begin{align}\label{c1}
w_n\to 0\qquad \text{ in $L^q(\R^N)$ for all $q\in\left(2,\frac{2N}{N-2s}\right).$ }
\end{align}
Arguing by contradiction, assume that $w_n$ does not converge to 0 in $L^q(\R^N)$ for some $q\in(2,\frac{2N}{N-2s}).$ Since $w_n\to 0$ in $L_{loc}^2(\R^N)$, we have, by Lions Lemma (see \cite[Lemma 2.4]{S13}), that there are $(x_n)\subset \R^N$ such that $|x_n|\to \infty$ as $n\to\infty$ and $\delta>0$ such that
\begin{align*}
\int_{B_1(x_n)}|w_n|^2\, dy >2\delta\qquad \text{ for all }n\in\N.
\end{align*}
Note that
\begin{align*}
    \int_{B_1(x_n)}|w_n|^2\, dy
    =\int_{B_1(x_n)}|u_{R_n}|^2-2u_*u_{R_n}+|u_*|^2\, dy = \int_{B_1(x_n)}|u_{R_n}|^2\, dy+o(1)\qquad \text{ as }n\to\infty,
\end{align*}
because $u_*\in L^2(\R^N)$ and because there is $C_1>0$ such that $|u_{R_n}|^2_{2}<C_1$ for all $n\in\mathbb N$. Then, for all $n$ large enough,
\begin{align*}
 \int_{B_1(x_n)}|u_{R_n}|^2\geq \delta.
\end{align*}
By \cite[Corollary 1.2]{JW14}, we know that $u_{R_n}$ is symmetric and monotone decreasing in $x_i$ for every $i=1,\ldots,N.$ Consider the set of points connecting the ball $B_1(0)$ and $B_1(x_n)$, namely, 
\begin{align*}
L_n:=\{x\in \R^N\::\:x=tb_1+(1-t)b_2,\, b_1\in B_1(0), b_2\in B_1(x_n)\}.    
\end{align*}
Since $D$ satisfies $(\mathcal S),$ we have that $L_n\subset D_{R_n}.$  Fix $k\in\mathbb N$ such that $\delta k >C_1$. By taking $n$ sufficiently large, we have that $|x_n|\geq 2k+1$. For $j=1,\ldots,k$, let $\zeta_j:=2j\frac{x_n}{|x_n|}$, then $B_1(\zeta_i)\subset L_n$ and, by monotonicity,
\begin{align*}
C_1>|u_{R_n}|_{2}^2\geq \int_{L_n}|u_{R_n}|^2\, dx\geq \sum_{j=1}^k \int_{B_1(\zeta_k)}|u_{R_n}|^2\, dx\geq k\int_{B_1(x_n)}|u_{R_n}|^2\, dx\geq \delta k >C_1.
\end{align*}
This yields a contradiction and~\eqref{c1} follows.  Then~\eqref{c0} implies that $|u_*|_{p+1}=|Q|_{p+1}> 0$.  Furthermore, by~\eqref{cm1},~\eqref{c1}, and Lemma~\ref{lem:ceps}, we have that $u_*$ is a (least-energy) weak solution of~\eqref{PQlambda}. By uniqueness, we have that $u_*$ must be a translation of the unique radial least-energy solution $Q$ of~\eqref{PQlambda}.  But using that $u_*$ is symmetric and monotone in $x_i$, we arrive at the equality $u_*=Q.$  Since the limit does not depend on the sequence $R_n$, we obtain the claim.
\end{proof}

\subsection{Nondegeneracy and uniqueness} \label{subsec:UniqNondegLarge}

We are ready to show Theorem~\ref{theorem:uniquenessLARGE}. Recall the definition of $D_R$ and $H^s_G(\mathbb R^N)$ given in~\eqref{dr:def} and~\eqref{HsG:def} respectively. 

\begin{proof}[Proof of Theorem~\ref{theorem:uniquenessLARGE}]
\emph{Nondegeneracy:}  We argue the nondegeneracy first,  by contradiction.  
Assume that there exists a sequence 
$R_n\rightarrow +\infty$ as $n\rightarrow +\infty$ and a sequence of least-energy solutions $u_{R_n}$ of ~\eqref{Peps} in $\Omega=D_{R_n}$ such that the linearized problem~\eqref{linearizedEquationSection} at $u_{R_n}$ admits a  non trivial solution $v_n\in \cH^s_0( D_{R_n})$. Since $u_{R_n}$ has Morse index one (see~\eqref{MorseofLeastEnergy}), then $v_n$ is a second eigenfunction for the linearized eigenvalue problem~\eqref{eigenvalueProblematu}, and the  associated  eigenvalue is $\mu_2=0$.
By our symmetry result for second eigenfunctions (see Corollary~\ref{corollary: symmetryOfKernel_CylindricalDomains}) we know that $v_n$ is sign-changing and it belongs to $ H_{G}^s(\mathbb R^N)$. Since~\eqref{linearizedEquationSection} is a linear problem, without loss of generality we may assume that $\|v_n\|=1$ for all $n\in \N$. By the reflexivity of the space $H^s(\mathbb R^N)$, there exists $v_{\ast}\in H_{G}^s(\mathbb R^N)$
such that 
$v_n$ converges to $v_{\ast}$ weakly in $H^s(\mathbb R^N)$, strongly in $L^2_{loc}(\mathbb R^N),$ and pointwisely for a.e. $x\in\mathbb R^N$.  Furthermore, by the compactness of the embedding of $H^s_{G}(\mathbb R^N)\hookrightarrow L^{p+1}(\mathbb R^N)$ (see Lemma~\ref{prop:compact}), we deduce that
\begin{align}\label{ccvn}
v_n\rightarrow v_{\ast}\mbox{ strongly in } L^{p+1}(\mathbb R^N).    
\end{align}
Let us recall that 
\begin{equation}\label{unifBoundpu}
|u_{R_n}|_{p+1}
\leq C_1
\end{equation}
for some constant $C_1>0,$ by Lemma~\ref{lem:bds}.   Hence,  taking $v_n$ as a test function in the weak formulation of~\eqref{linearizedEquationSection},  using   H\"older inequality 
with exponents $\frac{p+1}{p-1}$ and $\frac{p+1}{2}$, and~\eqref{unifBoundpu}
we derive that
\begin{align*}
    1=\|v_n\|^2=p\int_{\mathbb R^N} 
u_{R_n}^{p-1}v_n^2
\leq p
|u_{R_n}|_{p+1}^{\frac{p-1}{p+1}}
|v_n|_{p+1}^{\frac{2}{p+1}}\leq C_2|v_n|_{p+1}^{\frac{2}{p+1}}
\end{align*}
for some $C_2>0$.  Then, by~\eqref{ccvn}, $|v_{\ast}|_{p+1}+o(1)=|v_n|_{p+1}\geq C_2^{-1}>0$ and, as a consequence, $v_\ast\not\equiv 0$. 

By Proposition~\ref{prop:utoQ}, we know that $u_{R_n}\to Q$ in $H^s(\R^N)$ as $n\to +\infty$ , hence passing to the limit into the weak formulation of~\eqref{linearizedEquationSection} we obtain that $v_{\ast}$ is a   nontrivial solution of 
\begin{equation}
\label{LIMITlinearizedproblem}
(-\Delta)^{s} v_{\ast}+\lambda v_{\ast} = pQ^{p-1}v_{\ast} \qquad {\text{ in }} \mathbb R^N.
\end{equation}
In particular, $v_{\ast}\in \operatorname{ker}(L_+)$ and it is symmetric with respect to $x_i$, $i=1,\ldots, N$ and nontrivial.
This is a contradiction with the nondegeneracy of the least-energy solution $Q$ (see Proposition~\ref{prop:nongeneracyLimitProblem}), indeed~\eqref{kernelLimitProblem} implies in particular the antisymmetry with respect to $x_j$, for some $j=1,\ldots, N$, for any function in $\operatorname{ker}(L_+)$.

\medskip

\emph{Uniqueness:}  Now we argue the uniqueness, again by contradiction. Assume that there is a sequence $R_n\to +\infty$ and two different least-energy solutions $u_{R_n}$ and $\widetilde u_{R_n}$ of $\eqref{Peps}$ in $\Omega=D_R$, with $R={R_n}$ for every $n\in \mathbb N$. Let $w_n:=\frac{u_{R_n}-\widetilde u_{R_n}}{\|u_{R_n}-\widetilde u_{R_n}\|}$. 
Then, $w_n\in \cH^s_0( D_{R_n})$ solves weakly 
\begin{align}\label{wn:eq}
    (-\Delta)^s w_n+\lambda w_n = c_n w_n\quad \text{ in } D_{R_n},
\end{align}
where $c_n:=\int_0^1 p|tu_{R_n}+(1-t)\widetilde u_{R_n}|^{p-1}\, dt.$
Furthermore, since by \cite[Corollary 1.2]{JW14} $u_{R_n}$ and $\widetilde u_{R_n}$ are symmetric with respect to $x_i$, for all $i=1,\ldots, N$,  $w_{n}$ is symmetric with respect to $x_i$, for all $i=1,\ldots, N$.\\
By definition $\|w_n\|=1,$ hence there is $w_*\in H^s(\R^N)$ such that $w_n\weakto w_*$ as $n\to\infty$.  
Therefore $w_*$ inherits the same symmetry of $w_n$, namely it is symmetric with respect to $x_i$, for all $i=1,\ldots, N$.\\
By Proposition~\ref{prop:utoQ}, we know that $u_{R_n},\widetilde u_{R_n}\to Q$ in $H^s(\R^N)$ as $n\to +\infty$, in particular, by dominated convergence,
\begin{align*}
\lim_{n\to\infty}c_n(x)=\lim_{n\to\infty}\int_0^1 p|tu_{R_n}(x)+(1-t)\widetilde u_{R_n}(x)|^{p-1}\, dt=p|Q(x)|^{p-1}=:c(x)\qquad \text{ for all }x\in \R^N.
\end{align*}
Then, by weak convergence, $w_*$ solves
\begin{align*}
(-\Delta)^s w_* + \lambda w_* = p|Q|^{p-1}w_*\qquad \text{ in }\R^N.
\end{align*}
By the nondegeneracy of $Q$ (see Proposition~\ref{prop:nongeneracyLimitProblem}), we have that $w_*$ is a linear combination of  the functions $\partial_{x_i}Q$ for $i=1,\ldots,N.$ Since $Q$ is radially symmetric, this implies that $\partial_{x_i}Q$ is antisymmetric with respect to $x_i$. By the symmetry propertis of  $w_*$, this implies that $w_*$ is orthogonal to $\partial_{x_i}Q$ for every $i=1,\ldots,N,$ in $L^2$-sense and therefore
\begin{align}\label{c2}
w_*\equiv 0.    
\end{align}

Testing~\eqref{wn:eq} with $w_n$ and using Hölder's inequality, we obtain that
\begin{align*}
1=\|w_n\|^2=\int_{\R^N}c_n w_n^2\, dx
\leq |w_n|_{p+1}^2 |c_n|_{\frac{p+1}{p-1}}.
\end{align*}

By the definition of $c_n$, we have that
\begin{align*}
|c_n|_{\frac{p+1}{p-1}}^{\frac{p+1}{p-1}}
\leq p^{\frac{p+1}{p-1}}2^{p+1}\int_ D|u_{R_n}|^{p+1}+|\widetilde u_{R_n}|^{p+1}\, dx=p^{\frac{p+1}{p-1}}2^{p+2}|Q|_{\frac{p+1}{p-1}}^{\frac{p+1}{p-1}}+o(1)
\end{align*}
as $n\to\infty.$ Then, there is $\delta>0$ such that $|w_*|_{p+1}^2+o(1)=|w_n|_{p+1}^2>\delta$ as $n\to \infty$, where we use the compact embedding $\cH^s_{G}(\R^N) \hookrightarrow L^{p+1}(\R^N)$ (see Lemma~\ref{prop:compact}).  This yields a contradiction with~\eqref{c2} and the claim follows.
\end{proof}

\begin{remark}\label{rmk:nondegeneracyImpliesUniquenes}
In our proof of Theorem~\ref{theorem:uniquenessLARGE} uniqueness and nondegeneracy have been proved separately and independently. In particular we have obtained the uniqueness of $u$ with a direct argument, which does not make use of the symmetry result for the linearized  problem contained  in Theorem~\ref{sym}, but exploits the uniqueness and nondegeneracy of the ground state solution in $\mathbb R^N$.
Theorem~\ref{sym} is instead at the core of the proof of the nondegeneracy of $u$.\\
We stress that  the uniqueness of $u$ could also be derived alternatively from the nondegeneracy, following similar arguments as in the local case  (see \cite{lin1994,DamascelliGrossiPacellaAIHP1999}).
Indeed we already know that uniqueness holds for $p$ close to $1$  (see \cite{DIS22}), so thanks to the nondegeneracy we could apply a continuation argument in $p$ based on the implicit function theorem and recover the uniqueness  result in the full range of $p$'s (see the arguments in \cite[proof of Theorem 1.4]{FW23preprint}).
\end{remark}
\begin{remark}
    Theorem~\ref{theorem:uniquenessLARGE} holds for least-energy solutions and not for any non-negative solution $u$ of~\eqref{Peps}.  This is due to the fact that the uniqueness and nondegeneracy result in \cite{FLS16} is known only for ground state solutions of the problem in $\mathbb R^N$.
If the results in \cite{FLS16} were known to hold for \emph{any} positive $H^s(\mathbb R^N)$ solution, then
our direct proof for the uniqueness (see Remark~\ref{rmk:nondegeneracyImpliesUniquenes}) can be adjusted to imply \emph{uniqueness  for any non-negative solution} $u$ of~\eqref{Peps}, provided the symmetric domain is large enough. 
   Once uniqueness is established, then the nondegeneracy also follows by Theorem~\ref{theorem:uniquenessLARGE} (indeed the unique solution would then be the least-energy one).
\end{remark}

\begin{flushleft}
\textbf{Abdelrazek Dieb}\\
Department of Mathematics, Faculty of Mathematics and computer science \\
University Ibn Khaldoun of Tiaret\\
Tiaret 14000, Algeria\\
and\\
Laboratoire d’Analyse Nonlinéaire et Mathématiques Appliquées\\
Université Abou Bakr Belkaïd, Tlemcen, Algeria\\
\texttt{abdelrazek.dieb@univ-tiaret.dz}
\vspace{.3cm}

\textbf{Isabella Ianni}\\
Dipartimento di Scienze di Base e Applicate per l’Ingegneria\\
Sapienza Universit\`a di Roma\\
Via Scarpa 16, 00161 Roma, Italy\\
\texttt{isabella.ianni@uniroma1.it} 
\vspace{.3cm}

\textbf{Alberto Saldaña}\\
Instituto de Matemáticas\\
Universidad Nacional Autónoma de México\\
Circuito Exterior, Ciudad Universitaria\\
04510 Coyoacán, Ciudad de México, Mexico\\
\texttt{alberto.saldana@im.unam.mx} 
\vspace{.3cm}
\end{flushleft}
\end{document}